\documentclass[a4paper,11pt]{amsart}

\usepackage{amsmath, amsthm, amsfonts, amssymb, xcolor}
\usepackage{enumerate}

\usepackage[nosort,nocompress]{cite}

\usepackage[linktocpage=true,bookmarks=false,hyperfootnotes=false,colorlinks,
    linkcolor={red!60!black},
    citecolor={blue!50!black},
    urlcolor={blue!80!black}]{hyperref}

\renewcommand{\eqref}[1]{\hyperref[#1]{(\ref{#1})}}

\numberwithin{equation}{section}

\newtheorem{letterthm}{Theorem}

\newtheorem{lettercor}[letterthm]{Corollary}

\newtheorem{thm}{Theorem}[section]
\newtheorem{lem}[thm]{Lemma}

\newtheorem{cor}[thm]{Corollary}
\newtheorem{prop}[thm]{Proposition}

\theoremstyle{definition}

\newtheorem{remark}[thm]{Remark}
\newtheorem{example}[thm]{Example}

\DeclareMathOperator*{\bigast}{{\text{\huge $\ast$}}}

\newcommand{\R}{\mathbf{R}}
\newcommand{\C}{\mathbf{C}}
\newcommand{\Z}{\mathbf{Z}}
\newcommand{\F}{\mathbf{F}}

\newcommand{\N}{\mathbf{N}}

\newcommand{\cP}{\mathcal{P}}
\newcommand{\cF}{\mathcal{F}}
\newcommand{\cZ}{\mathcal{Z}}
\newcommand{\cU}{\mathcal{U}}

\newcommand{\cM}{\mathcal{M}}

\newcommand{\cC}{\mathcal{C}}

\newcommand{\cn}{\mathfrak{n}}
\newcommand{\cm}{\mathfrak{m}}

\newcommand{\Ad}{\operatorname{Ad}}
\newcommand{\id}{\text{\rm id}}

\newcommand{\si}{\sigma}
\newcommand{\vphih}{\hat{\vphi}}
\newcommand{\psih}{\hat{\psi}}

\newcommand{\rL}{\mathord{\text{\rm L}}}
\newcommand{\rB}{\mathord{\text{\rm B}}}

\newcommand{\ap}{\mathord{\text{\rm ap}}}
\newcommand{\rT}{\mathord{\text{\rm T}}}
\newcommand{\rD}{\mathord{\text{\rm D}}}
\newcommand{\rE}{\mathord{\text{\rm E}}}

\newcommand{\Sd}{\mathord{\text{\rm Sd}}}

\newcommand{\core}{\mathord{\text{\rm c}}}
\newcommand{\rd}{\mathord{\text{\rm d}}}

\newcommand{\Tr}{\mathord{\text{\rm Tr}}}

\newcommand{\supp}{\mathord{\text{\rm supp}}}
\newcommand{\ot}{\otimes}
\newcommand{\ovt}{\mathbin{\overline{\otimes}}}

\newcommand{\al}{\alpha}
\newcommand{\dpr}{^{\prime\prime}}

\newcommand{\recht}{\rightarrow}

\newcommand{\vphi}{\varphi}

\newcommand{\eps}{\varepsilon}

\newcommand{\cS}{\mathcal{S}}
\newcommand{\Om}{\Omega}

\newcommand{\ri}{\text{\rm i}}
\newcommand{\Hcirc}{\overset{\circ}{H}}
\newcommand{\cA}{\mathcal{A}}
\newcommand{\mutil}{\widetilde{\mu}}

\begin{document}

\title[Classification of a family of non almost periodic free Araki--Woods factors]{Classification of a family of non almost periodic\\ free Araki--Woods factors}

\begin{abstract}
We obtain a complete classification of a large class of non almost periodic free Araki--Woods factors $\Gamma(\mu, m)\dpr$ up to isomorphism. We do this by showing that free Araki--Woods factors $\Gamma(\mu, m)\dpr$ arising from finite symmetric Borel measures $\mu$ on $\R$ whose atomic part $\mu_a$ is nonzero and not concentrated on $\{0\}$ have the joint measure class $\mathcal C(\bigvee_{k \geq 1} \mu^{\ast k})$ as an invariant. Our key technical result is a deformation/rigidity criterion for the unitary conjugacy of two faithful normal states. We use this to also deduce rigidity and classification theorems for free product von Neumann algebras.
\end{abstract}

\author{Cyril Houdayer}
\address{Laboratoire de Math\'ematiques d'Orsay \\ Universit\'e Paris-Sud \\ CNRS \\ Universit\'e Paris-Saclay \\ 91405 Orsay \\ France}
\email{cyril.houdayer@math.u-psud.fr}
\thanks{CH is supported by ERC Starting Grant GAN 637601}

\author{Dimitri Shlyakhtenko}
\address{Mathematics Department\\ UCLA \\ Los Angeles \\ CA 90095-1555 \\ United States}
\email{shlyakht@math.ucla.edu}
\thanks{DS is supported by NSF Grant DMS-1500035}

\author{Stefaan Vaes}
\address{KU Leuven \\ Department of Mathematics \\ Celestijnenlaan 200B \\ B-3001 Leuven \\ Belgium}
\email{stefaan.vaes@kuleuven.be}
\thanks{SV is supported by ERC Consolidator Grant 614195, and by long term structural funding~-- Methusalem grant of the Flemish Government}

\subjclass[2010]{46L10, 46L54, 46L36}

\keywords{Free Araki--Woods factors; Free product von Neumann algebras; Popa's deformation/rigidity theory; Type ${\rm III}$ factors}

\maketitle

\mbox{}\vspace{-6ex}

%%%%%%%%%%
\section{Introduction}

Free Araki-Woods factors are a free probability analog of the type III hyperfinite factors, just like free group factors are free probability analogs of the hyperfinite II$_1$ factor. The classification of hyperfinite type III factors has a beautiful history originating in the work of Powers \cite{Po67}. Through the works of Connes \cite{Co72}, Haagerup \cite{Ha85} and Krieger \cite{Kr75}, the classification question was ultimately reduced to the classification of ergodic actions of the additive group of real numbers $\R$, i.e., to classification of virtual subgroups of $\R$ in the sense of Mackey.

Following \cite{Sh96}, to every orthogonal representation $(U_t)_{t \in \R}$ of $\R$ on a real Hilbert space $H_\R$ is associated the free Araki--Woods factor $\Gamma(U,H_\R)\dpr$.
For almost periodic representations, i.e.\ when $(U_t)$ is a direct sum of finite dimensional representations, the free Araki--Woods factors were completely classified in \cite{Sh96} by Connes' $\Sd$ invariant \cite{Co74}, which is in this case equal to the subgroup of $(\R^\ast_+,\cdot)$ generated by the eigenvalues of $(U_t)$.
%~: two such free Araki--Woods factors are isomorphic if and only if they have the same $\Sd$ invariant in the sense of Connes \cite{Co74}.
Beyond the almost periodic case, the classification of free Araki--Woods factors is a very intriguing open problem and there is not even a conjectural classification statement. So far, one could only distinguish between families of non almost periodic free Araki--Woods factors by computing their invariants, like Connes' $\tau$-invariant (see \cite{Sh97b,Sh02}), or by structural properties of their continuous core (see \cite{Sh02, Ho08b, Ha15}). In this paper, we prove the first complete classification theorem for a large family of non almost periodic free Araki--Woods factors.

Orthogonal representations of $\R$ are classified by their spectral measure and multiplicity function. So, to any finite symmetric Borel measure $\mu$ on $\R$ and to any symmetric Borel multiplicity function $m : \R \to \N \cup \{+\infty\}$ (that we always assume to satisfy $m \geq 1$ $\mu$-almost everywhere), we associate the free Araki--Woods factor $\Gamma(\mu, m)\dpr$, which comes equipped with the free quasi-free state $\varphi_{\mu, m}$. The almost periodic case corresponds to $\mu$ being an atomic measure and then, by \cite{Sh96}, $\Gamma(\mu_1,m_1)\dpr$ is isomorphic with $\Gamma(\mu_2,m_2)\dpr$ if and only if the sets of atoms of $\mu_1$ and $\mu_2$ generate the same subgroup of $(\R,+)$.

In this paper, we fully classify the free Araki--Woods factors in the case where the atomic part $\mu_a$ is nonzero and not concentrated on $\{0\}$ and where the continuous part $\mu_c$ satisfies $\mu_c \ast \mu_c \prec \mu_c$. We find in particular that in that case, the free Araki--Woods factor does not depend on the multiplicity function $m$. But we also show that in other cases, $\Gamma(\mu, m)\dpr$ does depend on $m$.

In order to state our main results, we first introduce some terminology.
%We call measure on $\R$ every $\sigma$-finite measure on the Borel sets of $\R$.
For every $\sigma$-finite Borel measure $\mu$ on $\R$, we denote by $\cC(\mu)$ the \emph{measure class} of $\mu$, defined as the set of all Borel sets $U \subset \R$ with $\mu(U)=0$. Note that $\cC(\mu) = \cC(\nu)$ if and only if $\mu \sim \nu$, while $\cC(\mu) \subset \cC(\nu)$ if and only if $\nu \prec \mu$. For any sequence of measures $(\mu_k)_{k \in \N}$, we denote by $\bigvee_{k \in \N} \mu_k$ any measure with the property that $\mathcal C(\bigvee_{k \in \N} \mu_k) = \bigcap_{k \in \N} \mathcal C(\mu_k)$. We denote by $\mu = \mu_c + \mu_a$ the unique decomposition of a measure $\mu$ as the sum of a continuous and an atomic measure.

We show that free Araki--Woods factors $\Gamma(\mu, m)\dpr$ arising from finite symmetric Borel measures $\mu$ on $\R$ whose atomic part $\mu_a$ is nonzero and not concentrated on $\{0\}$ have the  joint measure class $\mathcal C(\bigvee_{k \geq 1} \mu^{\ast k})$ as an invariant. More precisely, we obtain the following result.

\begin{letterthm}\label{thmA}
Let $\mu, \nu$ be finite symmetric Borel measures on $\R$ and $m, n : \R \to \N \cup \{+\infty\}$ symmetric Borel multiplicity functions. Assume that $\nu$ has at least one atom not equal to $0$.

If the free Araki--Woods factors $\Gamma(\mu, m)\dpr$ and $\Gamma(\nu, n)\dpr$ are isomorphic, then there exists an isomorphism that preserves the free quasi-free states. In particular, the joint measure classes $\mathcal C(\bigvee_{k \geq 1} \mu^{\ast k})$ and $\mathcal C(\bigvee_{k \geq 1} \nu^{\ast k})$ are equal.
\end{letterthm}

Denote by $\mathcal S(\R)$ the set of all finite symmetric Borel measures $\mu = \mu_c + \mu_a$ on $\R$ satisfying the following two properties:
\begin{itemize}
\item [$(\rm i)$] $\mu_c \ast \mu_c \prec \mu_c$ and
\item [$(\rm ii)$] $\mu_a \neq 0$ and $\supp(\mu_a) \neq \{0\}$.
\end{itemize}
Denote by $\Lambda(\mu_a)$ the countable subgroup of $\R$ generated by the atoms of $\mu_a$ and by $\delta_{\Lambda(\mu_a)}$ a finite atomic measure on $\R$ whose set of atoms equals $\Lambda(\mu_a)$.

Combining our Theorem \ref{thmA} with the isomorphism Theorem \ref{thm.isom} below, we obtain a complete classification of the free Araki--Woods factors arising from measures in $\mathcal S(\R)$. Here and elsewhere in this paper, we call \emph{isomorphism} between von Neumann algebras $M$ and $N$ any bijective $\ast$-isomorphism. Even when $M$ and $N$ are equipped with distinguished faithful normal states, isomorphisms are not assumed to preserve these states.

\begin{lettercor}\label{corB}
The set of free Araki--Woods factors
$$\left\{ \Gamma(\mu, m)\dpr : \mu \in \mathcal S(\R) \text{ and } m : \R \to \N \cup \{+\infty\} \text{ is a symmetric Borel multiplicity function}\right\}$$
is exactly classified, up to isomorphism, by the countable subgroup $\Lambda(\mu_a)$ and the measure class $\mathcal C(\mu_c \ast \delta_{\Lambda(\mu_a)})$.
\end{lettercor}

Note that the measure class $\cC(\mu_c \ast \delta_{\Lambda(\mu_a)})$ equals the set of Borel sets $U \subset \R$ satisfying $\mu_c(x+U)=0$ for all $x \in \Lambda(\mu_a)$ and, in particular, does not depend on the choice of $\delta_{\Lambda(\mu_a)}$.

The family $\mathcal S(\R)$ is large and provides many nonisomorphic free Araki--Woods factors having the same Connes' invariants and in particular the same $\tau$-invariant, see Example \ref{ex.many-measure}. Note that previously, only two non almost periodic free Araki--Woods factors having the same $\tau$-invariant could be distinguished, see \cite[Theorem 5.6]{Sh02}.

Combining our Corollary \ref{corB} with \cite[Theorem A]{HI15}, we also obtain a complete classification for tensor products of free Araki--Woods factors arising from measures in $\mathcal S(\R)$.

We then show that free Araki--Woods factors $\Gamma(\mu, m)\dpr$ arising from continuous finite symmetric Borel measures $\mu$ on $\R$ have all their centralizers amenable, i.e.\ the centralizer of any faithful normal state is amenable. More generally, we obtain the following result.

\begin{lettercor}\label{corC}
Let $\mu$ be any finite symmetric Borel measure on $\R$ and $m : \R \to \N \cup \{+\infty\}$ any symmetric Borel multiplicity function. The free Araki--Woods factor $\Gamma(\mu, m)\dpr$ has all its centralizers amenable if and only if the atomic part $\mu_a$ of $\mu$ is either zero or is concentrated on $\{0\}$ with $m(0)=1$.
\end{lettercor}

By \cite{Ho08a}, all free Araki--Woods factors $M$ satisfy Connes' bicentralizer conjecture (see \cite{Co80}) and thus, by \cite[Theorem 3.1]{Ha85}, admit faithful normal states $\vphi$ such that $M^\vphi \subset M$ is an irreducible subfactor. So, having all centralizers amenable is the smallest centralizers can be in general.

In the setting of Corollary \ref{corC} and under the additional assumption that the Fourier transform of the continuous finite symmetric Borel measure $\mu_c$ vanishes at infinity, it was shown in \cite[Theorem 1.2]{Ho08b} that the continuous core of the corresponding free Araki--Woods factor $\Gamma(\mu, m)\dpr$ is {\em solid} (see \cite{Oz03}), meaning that the relative commutant of any diffuse subalgebra that is the range of a faithful normal conditional expectation is amenable. Any type ${\rm III_1}$ factor whose continuous core is solid has all its centralizers amenable. Observe that there are many free Araki--Woods factors arising in Corollary \ref{corC} whose Connes' $\tau$-invariant (see \cite{Co74}) is not the usual topology on $\R$. In particular, these free Araki--Woods factors have a continuous core that is not {\em full} (see \cite{Co74, Sh02}) and hence not solid (see \cite[Proposition 7]{Oz03} with $\mathcal N_0 = \mathcal M$). Therefore, Corollary \ref{corC} provides many new examples of type ${\rm III_1}$ factors whose centralizers are all amenable.

The following is an immediate consequence of Corollary \ref{corC}.

\begin{lettercor}\label{corD}
Let $\lambda$ be the Lebesgue measure on $\R$. Then $\Gamma(\lambda+\delta_0,1)\dpr \not\cong \Gamma(\lambda+\delta_0,2)\dpr$.
So in certain cases, the isomorphism class of $\Gamma(\mu,m)\dpr$ depends on the multiplicity function $m$.
\end{lettercor}

Our main technical tool to prove the results mentioned so far is a deformation/rigidity criterion for the unitary conjugacy of two faithful normal states on a von Neumann algebra $M$. In Corollary \ref{cor.key} below, we prove that a corner of the state $\psi$ is unitarily conjugate with a corner of the state $\vphi$ if and only if in the continuous core $\core(M)$, there is a Popa intertwining bimodule (in the sense of \cite{Po02,Po03}) between the canonical subalgebras $\rL_\psi(\R)$ and $\rL_\vphi(\R)$ of $\core(M)$ given by realizing $\core(M)$ as respectively $M \rtimes_{\sigma^\psi} \R$ and $M \rtimes_{\sigma^\vphi} \R$ (see Section \ref{sec.prelim} for details).

More generally, when $P \subset M$ is a von Neumann subalgebra that is the range of a faithful normal conditional expectation $\rE_P : M \recht P$, we provide in Theorem \ref{thm-key} below a deformation/rigidity criterion describing when a state $\psi$ on $M$ has a corner that is unitarily conjugate with a corner of a state of the form $\theta \circ \rE_P$. Applying this criterion to a free product von Neumann algebra $M$, we obtain the following complete characterization of when $M$ has all its centralizers amenable.

\begin{letterthm}\label{thmE}
For $i=1,2$, let $(M_i,\vphi_i)$ be a von Neumann algebra with a faithful normal state. Denote by $(M,\vphi) = (M_1,\vphi_1)\ast (M_2,\vphi_2)$ their free product. Then $M$ has all its centralizers amenable if and only if both $M_1$ and $M_2$ have all their centralizers amenable and $M^\vphi$ is amenable.
\end{letterthm}

Finally, we use the same criterion, in combination with methods of \cite{Sh97a}, to prove the following classification result for a free product with a free Araki--Woods factor.

\begin{letterthm}\label{thmF}
Let $\mu$ be a continuous finite symmetric Borel measure on $\R$. Fix the free Araki--Woods factor $(M,\vphi) = (\Gamma(\mu,+\infty)\dpr , \vphi_{\mu,+\infty})$, where $+\infty$ denotes the multiplicity function equal to $+\infty$ everywhere.
\begin{itemize}
\item [$(\rm i)$] If $(A,\tau)$ and $(B,\tau)$ are nonamenable ${\rm II_1}$ factors with their tracial states, then the free products $(M,\vphi) \ast (A,\tau)$ and $(M,\vphi) \ast (B,\tau)$ are isomorphic (not necessarily in a state preserving way) if and only if there exists a $t > 0$ such that $A \cong B^t$.

\item [$(\rm ii)$] If $(A_i,\psi_i)$, $i=1,2$, are full type ${\rm III}$ factors with almost periodic states having a factorial centralizer $A_i^{\psi_i}$, then the free products $(M,\vphi) \ast (A_1,\psi_1)$ and $(M,\vphi) \ast (A_2,\psi_2)$ are isomorphic (again, not necessarily in a state preserving way) if and only if $A_1 \cong A_2$.
\end{itemize}
\end{letterthm}

By Theorem \ref{thmF}, the free Araki--Woods factors $\Gamma((\lambda,+\infty) + (\delta_0,m))\dpr$ are isomorphic for all $2 \leq m < +\infty$, but the question whether these are isomorphic with $\Gamma((\lambda,+\infty) + (\delta_0,+\infty))\dpr$ is equivalent with the free group factor problem $\rL(\F_m) \cong^? \rL(\F_\infty)$.

{\bf Acknowledgment.} We are grateful to the Mittag-Leffler Institute for their hospitality during the program {\it Classification of operator algebras: complexity, rigidity, and dynamics}, where part of the work on this paper was done.

{\bf Disclaimer.} Some of the results in this paper, in particular Corollary \ref{corD}, Theorem \ref{thm.isom} and Proposition \ref{prop.core-lambda-plus-delta-0}, were obtained in the preprint \cite{Sh03} by the second named author. That preprint will remain unpublished and has been incorporated in this article.

%{\renewcommand{\contentsname}{Contents\vspace{-1.5ex}}\setlength{\parskip}{0.1ex}
%\tableofcontents}

{\setlength{\parskip}{0.3ex}
\tableofcontents}

\section{Preliminaries}\label{sec.prelim}

For any von Neumann algebra $M$, we denote by $\mathcal U(M)$ its group of unitaries. For any (possibly unbounded) positive selfadjoint closed operator $A$ on a separable Hilbert space $H$, we denote by $\mathcal C(A)$ the measure class on $\R$ of the spectral measure of $\log(A)$, i.e.\ the set of all Borel sets $\cU \subset \R$ such that the spectral projection $1_{\cU}(\log(A))$ equals $0$. Here $1_{\cU}$ denotes the function that is equal to $1$ on $\cU$ and equal to $0$ elsewhere. Note that we can always choose a measure $\mu$ on $\R$ such that $\cC(A) = \cC(\mu)$.

\subsection*{Free Araki--Woods factors}

Following \cite{Sh96}, we associate to every orthogonal representation $(U_t)_{t \in \R}$ of $\R$ on the real Hilbert space $H_\R$ the free Araki--Woods factor $\Gamma(H_\R,U_t)\dpr$, equipped with the free quasi-free state $\vphi_U$.

Denoting by $H = H_\R + {\rm i} H_\R$ the complexification of $H_\R$, define the positive nonsingular operator $\Delta$ on $H$ such that $U_t = \Delta^{\ri t}$ for all $t \in \R$. Also define the anti-unitary operator $J : H \recht H : J(\xi + {\rm i} \eta) = \xi - {\rm i} \eta$ for all $\xi,\eta \in H_\R$. Then, $J \Delta J = \Delta^{-1}$. Therefore, the measure class $\mathcal C(\Delta)$ of the spectral measure of $\log(\Delta)$ and the multiplicity function $m : \R \recht \N \cup \{+\infty\}$ of $\log(\Delta)$ are symmetric. This measure class and multiplicity function completely classify orthogonal representations of $\R$ on separable real Hilbert spaces. We therefore use the notation $(\Gamma(\mu,m)\dpr,\vphi_{\mu,m})$ to denote the free Araki--Woods factor and its free quasi-free state associated with the unique orthogonal representation with spectral invariant $(\mu,m)$.

\subsection*{Background on $\sigma$-finite von Neumann algebras}

Let $M$ be any $\sigma$-finite von Neumann algebra with predual $M_\ast$ and $\varphi \in M_\ast$ any faithful normal state. We denote by $\sigma^\varphi$ the modular automorphism group of the state  $\varphi$ defined by the formula $\sigma_t^\varphi = \Ad(\Delta_\varphi^{{\rm i}t})$ for all $t \in \R$.  The {\em centralizer} $M^\varphi$ of the state $\varphi$ is by definition the fixed point algebra of $(M, \sigma^\varphi)$. The {\em continuous core} of $M$ with respect to $\varphi$, denoted by $\core_\varphi(M)$, is the crossed product von Neumann algebra $M \rtimes_{\sigma^\varphi} \R$. The natural inclusion $\pi_\varphi: M \to \core_\varphi(M)$ and the unitary representation $\lambda_\varphi: \R \to \core_\varphi(M)$ satisfy the {\em covariance} relation
$$
  \lambda_\varphi(t) \pi_\varphi(x) \lambda_\varphi(t)^*
  =
  \pi_\varphi(\sigma^\varphi_t(x))
  \quad
  \text{ for all }
  x \in M \text{ and all } t \in \R.
$$
Put $\rL_\varphi (\R) := \lambda_\varphi(\R)\dpr$. There is a unique faithful normal conditional expectation $\rE_{\rL_\varphi (\R)}: \core_{\varphi}(M) \to \rL_\varphi(\R)$ satisfying $\rE_{\rL_\varphi (\R)}(\pi_\varphi(x) \lambda_\varphi(t)) = \varphi(x) \lambda_\varphi(t)$ for all $x \in M$ and all $t \in \R$. The faithful normal semifinite weight defined by $f \mapsto \int_{\R} \exp(-s)f(s) \, \rd s$ on $\rL^\infty(\R)$ gives rise to a faithful normal semifinite weight $\Tr_\varphi$ on $\rL_\varphi(\R)$ via the Fourier transform. The formula $\Tr_\varphi = \Tr_\varphi \circ \rE_{\rL_\varphi (\R)}$ extends it to a faithful normal semifinite trace on $\core_\varphi(M)$.

Because of Connes' Radon--Nikodym cocycle theorem \cite[Th\'eor\`eme 1.2.1]{Co72} (see also \cite[Theorem VIII.3.3]{Ta03}), the semifinite von Neumann algebra $\core_\varphi(M)$ together with its trace $\Tr_\varphi$ does not depend on the choice of $\varphi$ in the following precise sense. If $\psi \in M_\ast$ is another faithful state, there is a canonical isomorphism
$\Pi_{\varphi,\psi} : \core_\psi(M) \to \core_{\varphi}(M)$ of $\core_\psi(M$ onto $\core_{\varphi}(M)$ such that $\Pi_{\varphi,\psi} \circ \pi_\psi = \pi_\varphi$ and $\Tr_\varphi \circ \Pi_{\varphi,\psi} = \Tr_\psi$. Note however that $\Pi_{\varphi,\psi}$ does not map the subalgebra $\rL_\psi(\R) \subset \core_\psi(M)$ onto the subalgebra $\rL_\varphi(\R) \subset \core_\varphi(M)$ (and hence we use the symbol $\rL_\varphi(\R)$ instead of the usual $\rL(\R)$). We have $\Pi_{\varphi, \psi}(\lambda_\psi(t)) = \pi_\varphi(w_t) \lambda_\varphi(t)$ for every $t \in \R$, where $w_t = [\rD \psi : \rD \varphi]_t$ is Connes' Radon--Nikodym cocycle between $\psi$ and $\varphi$.

\begin{lem}\label{lem-measure}
Let $M$ be any $\sigma$-finite von Neumann algebra and $\varphi \in M_\ast$ any faithful state such that $M^\varphi$ is a ${\rm II_1}$ factor. Let $M^\varphi \subset P \subset M$ be any intermediate von Neumann subalgebra that is globally invariant under the modular automorphism group $\sigma^\varphi$. Denote by $\rE_P : M \to P$ the unique $\varphi$-preserving conditional expectation and write $M \ominus P := \ker(\rE_P)$. Let $p \in M^\varphi$ be any nonzero projection and put $\varphi_p := \frac{\varphi(p \, \cdot \, p)}{\varphi(p)} \in (pMp)_\ast$.

Then we have
$$\mathcal C(\Delta_\varphi) = \mathcal C(\Delta_{\varphi_p}) \quad \text{and} \quad \mathcal C(\Delta_\varphi |_{\rL^2(M) \ominus \rL^2(P)}) = \mathcal C(\Delta_{\varphi_p} |_{\rL^2(pMp) \ominus \rL^2(pPp)}).$$
\end{lem}

\begin{proof}
Fix a standard representation $M \subset \rB(H)$ and denote by $\xi_\vphi \in H$ the canonical unit vector that implements $\vphi$. For every $x \in M$, denote by $\mu^\varphi_{x}$ the unique finite Borel measure on $\R$ that satisfies
$$\varphi(x^*\sigma_t^\varphi(x)) = \langle \Delta_\varphi^{{\rm i}t} (x\xi_\varphi), x\xi_\varphi\rangle = \int_{\R} \exp({\rm i}st) \, \rd \mu_x^\varphi(s) \quad \text{for all } t\in \R.$$
For any Borel subset $U \subset \R$, we have $U \in \mathcal C(\Delta_\varphi)$ (resp.\ $U \in \mathcal C(\Delta_\varphi |_{\rL^2(M) \ominus \rL^2(P)})$) if and only if $\mu_x^\varphi(U) = 0$ for all $x \in M$ (resp.\ for all $x \in M \ominus P$). Since $pMp \subset M$ and $p(M \ominus P)p \subset M \ominus P$, it is clear that $\mathcal C(\Delta_{\varphi}) \subset \mathcal C(\Delta_{\varphi_p})$ (resp.\ $\mathcal C(\Delta_\varphi |_{\rL^2(M) \ominus \rL^2(P)}) \subset \mathcal C(\Delta_{\varphi_p} |_{\rL^2(pMp) \ominus \rL^2(pPp)})$). It remains to prove that $\mathcal C(\Delta_{\varphi_p}) \subset \mathcal C(\Delta_{\varphi})$ (resp.\ $C(\Delta_{\varphi_p} |_{\rL^2(pMp) \ominus \rL^2(pPp)}) \subset \mathcal C(\Delta_{\varphi} |_{\rL^2(M) \ominus \rL^2(P)})$).

Up to shrinking $p \in M^\varphi$ if necessary and since we have $p_2 M p_2 \subset p_1 M p_1$ and $p_2 (M \ominus P) p_2 \subset p_1 (M \ominus P)p_1$ whenever $p_2 \leq p_1$ with $p_1, p_2$ nonzero projections in $M^\varphi$, we may assume without loss of generality that $\varphi(p) = m^{-1}$ with $m \in \N$. Since $M^\varphi$ is a ${\rm II_1}$ factor, we may find partial isometries $u_1, \dots , u_m \in M^\varphi$ such that $u_1 = p$, $u_j^*u_j = p$ for all $1 \leq j \leq m$ and $\sum_{j = 1}^m u_j u_j^* = 1$.

Let $x \in M$ (resp.\ $x \in M \ominus P$). Since $\varphi(x^*\sigma_t^\varphi(x)) = \varphi(p)\sum_{i, j = 1}^m \varphi_p((u_j^*x u_i)^*\sigma_t^{\varphi_p}(u_j^*xu_i))$ for all $t \in \R$, we have $\mu_x^\varphi = \varphi(p)\sum_{i, j = 1}^m \mu_{u_j^* xu_i}^{\varphi_p}$. If $x \in M \ominus P$, then $u_j^* x u_i \in p(M \ominus P)p$ for all $1\leq i, j \leq m$. This implies that $\mathcal C(\Delta_{\varphi_p}) \subset \mathcal C(\Delta_\varphi)$ (resp.\ $\mathcal C(\Delta_{\varphi_p} |_{\rL^2(pMp) \ominus \rL^2(pPp)}) = \mathcal C(\Delta_{\varphi} |_{\rL^2(M) \ominus \rL^2(P)})$) and finishes the proof.
\end{proof}

\subsection*{Popa's intertwining-by-bimodules}

Popa introduced his {\em intertwining-by-bimodules} theory in \cite{Po02, Po03}. In the present work, we make use of this theory in the context of semifinite von Neumann algebras. We introduce the following terminology. Let $M$ be any $\sigma$-finite semifinite von Neumann algebra endowed with a fixed faithful normal semifinite trace $\Tr$. Let $1_A$ and $1_B$ be any nonzero projections in $M$
and let $A\subset 1_AM1_A$ and $B\subset 1_BM1_B$ be any von Neumann subalgebras. Assume that $\Tr(1_A) < +\infty$ and that $\Tr|_{B}$ is semifinite.

We say that $A$ {\em embeds into} $B$ {\em inside} $M$ and write $A \prec_M B$ if there exist a projection $e \in A$, a finite trace projection $f \in B$, a nonzero partial isometry $v \in eMf$ and a unital normal homomorphism $\theta : eAe \to fBf$ such that $av = v \theta(a)$ for all $a \in eAe$. We use the following useful characterization \cite{Po02, Po03} (see also \cite[Lemma 2.2]{HR10}).

\begin{thm}\label{thm-intertwining}
Keep the same notation as above. Denote by $\rE_B : 1_BM1_B \to B$ the unique trace preserving conditional expectation. Then the following conditions are equivalent.
\begin{itemize}
\item [$(\rm i)$] $A \prec_M B$.
\item [$(\rm ii)$] There exists no net $(w_i)_{i \in I}$ of unitaries in $\mathcal U(A)$ such that $\lim_i \|\rE_{B}(y^*w_i x)\|_2 =  0$ for all $x,y\in 1_AM1_B$.
\end{itemize}
\end{thm}

\section{A criterion for the unitary conjugacy of states}

Recall that for any von Neumann algebra $N$, any $\theta \in N_\ast$ and any $a, b \in N$, we define $(a\theta b)(y) := \theta(b y a)$ for every $y \in N$.

\begin{thm}\label{thm-key}
Let $M$ be a von Neumann algebra with a faithful normal state $\vphi \in M_*$ and $P \subset M$ a von Neumann subalgebra that is the range of a $\vphi$-preserving conditional expectation $\rE_P : M \recht P$. Let $\psi \in M_*$ be another faithful normal state and $q \in M^\psi$ a nonzero projection. Then the following statements are equivalent.
\begin{itemize}
\item [$(\rm i)$] There exists a nonzero finite trace projection $r \in \rL_\psi(\R)$ such that $$\Pi_{\varphi, \psi}(\rL_\psi(\R)qr) \prec_{\core_\varphi(M)} \core_\varphi(P) \; .$$
\item [$(\rm ii)$] There exist a faithful normal positive functional $\theta \in P_*$ and a nonzero partial isometry $v \in M$ such that $p = vv^* \in M^{\theta \circ \rE_P}$, $q_0 = v^* v \in q M^\psi q$ and $\psi q_0 = v^* (\theta \circ \rE_P) v$.
\end{itemize}
\end{thm}
\begin{proof}
Assume that $(\rm i)$ holds. Write $w_t = [\rD\psi: \rD\vphi]_t$. We claim that there exists a $\delta > 0$ and $x_1,\ldots,x_k \in q M$ such that
\begin{equation}\label{eq.what-we-get}
\sum_{i,j=1}^k \vphi\bigl( \; \rE_P(x_i^* \, w_t \, \si_t^\vphi(x_j) ) \; \rE_P(x_i^* \, w_t \, \si_t^\vphi(x_j) )^* \; \bigr) \geq \delta
\end{equation}
for all $t \in \R$. Assuming that the claim is false, we prove that $(\rm i)$ does not hold. Take a net $t_i \in \R$ such that
$$\lim_i \vphi\bigl( \; \rE_P(x^* \, w_{t_i} \, \si_{t_i}^\vphi(y) ) \; \rE_P(x^* \, w_{t_i} \, \si_{t_i}^\vphi(y) )^* \; \bigr) = 0$$
for all $x,y \in qM$. Using the $2$-norm $\|\,\cdot\,\|_2$ w.r.t.\ the canonical trace on $\core_\vphi(M)$, we get that for all finite trace projections $p,p' \in \rL_\vphi(\R)$ and all $x,y \in qM$,
\begin{align*}
\bigl\| p \, \rE_P(x^* \, w_{t_i} \, \si_{t_i}^\vphi(y) ) \, p' \bigr\|_2^2 & \leq \bigl\| p \, \rE_P(x^* \, w_{t_i} \, \si_{t_i}^\vphi(y) )\bigr\|_2^2 \\
& = \Tr(p) \, \vphi\bigl( \; \rE_P(x^* \, w_{t_i} \, \si_{t_i}^\vphi(y) ) \; \rE_P(x^* \, w_{t_i} \, \si_{t_i}^\vphi(y) )^* \; \bigr) \recht 0 \; .
\end{align*}
We then also get for all finite trace projections $p,p' \in \rL_\vphi(\R)$, all $s,s' \in \R$ and all $x,y \in M$ that
\begin{align*}
\bigl\| \rE_{\core_\vphi(P)}\bigl( p \, \lambda_\vphi(s)^* \, x^* \, & \Pi_{\vphi,\psi}(\lambda_\psi(t_i) q) \, y \, \lambda_\vphi(s') \, p' \bigr)  \bigr\|_2 \\
& = \bigl\| \lambda_\vphi(s)^* \, p \, \rE_P\bigl( (qx)^* w_{t_i} \sigma_{t_i}^\vphi(qy)\bigr) \, p' \, \lambda_\vphi(t_i + s')\|_2 \\
& = \bigl\| p \, \rE_P\bigl( (qx)^* w_{t_i} \sigma_{t_i}^\vphi(qy)\bigr) \, p' \|_2 \recht 0 \; .
\end{align*}
The linear span of such elements $x \lambda_\vphi(s) p$ is dense in $\rL^2(\core_\vphi(M),\Tr)$. So whenever $r \in \rL_\psi(\R)$ is a finite trace projection and $a,b \in \core_\vphi(M)$, we first approximate $\Pi_{\vphi,\psi}(r) \, b$ in $\|\, \cdot \, \|_2$ by a linear combination of $y \lambda_{\vphi}(s') p'$ and conclude that
$$\bigl\| \rE_{\core_\vphi(P)}\bigl( p \, \lambda_\vphi(s)^* \, x^* \,  \Pi_{\vphi,\psi}(\lambda_\psi(t_i) q r) \, b \bigr)\bigr\|_2 \recht 0$$
for all finite trace projections $p \in \rL_\vphi(\R)$ and all $s \in \R$, $x \in M$. We then approximate $\Pi_{\vphi,\psi}(r) \, a$ in $\|\, \cdot \, \|_2$ by a linear combination of $x \lambda_{\vphi}(s) p$ and conclude that
$$\bigl\| \rE_{\core_\vphi(P)}\bigl( a^* \, \Pi_{\vphi,\psi}(\lambda_\psi(t_i) q r) \, b \bigr)\bigr\|_2 \recht 0 \; .$$
Applying Theorem \ref{thm-intertwining}, we conclude that $(\rm i)$ does not hold. This concludes the proof of the claim.

Fix $\delta > 0$ and $x_1,\ldots,x_k$ such that \eqref{eq.what-we-get} holds for all $t \in \R$. Denote by $\langle M,e_P \rangle$ the basic construction for $P \subset M$, i.e.\ the von Neumann algebra acting on $\rL^2(M,\vphi)$ generated by $M$ acting by left multiplication and the orthogonal projection $e_P : \rL^2(M,\vphi) \recht \rL^2(P,\vphi)$. As in \cite[Section 2.1]{ILP96}, denote by $\vphih$ the canonical faithful normal semifinite weight on $\langle M,e_P \rangle$ characterized by
$$\sigma_t^{\vphih}(x e_P y) = \sigma_t^{\vphi}(x) e_P \sigma_t^{\vphi}(y) \quad\text{and}\quad \vphih(x e_P y) = \vphi(xy)$$
for all $x,y \in \R$. We have $\sigma_t^{\hat \vphi}(T) = \Delta_\vphi^{\ri t} T \Delta_\vphi^{-\ri t}$ for all $T \in \langle M,e_P \rangle$ and all $t \in \R$. In particular, $\sigma_t^{\vphih}(x) = \sigma_t^\vphi(x)$ for all $x \in M$. Denote by $\rT_M$ the unique faithful normal semifinite operator valued weight from $\langle M,e_P \rangle$ to $M$ such that $\vphih = \vphi \circ \rT_M$. Note that $\rT_M(e_P) = 1$.

Define $\psih = \psi \circ \rT_M$. So, $\psih$ is a faithful normal semifinite weight on $\langle M,e_P \rangle$ and by \cite[Theorem 4.7]{Ha77}, we have
$$[\rD \psih : \rD \vphih]_t = [\rD\psi : \rD\vphi]_t = w_t$$
for all $t \in \R$. In particular, we get that
\begin{equation}\label{eq.formula-sigma-psih}
\sigma_t^{\psih}(x e_P y) = w_t \sigma_t^\vphi(x) e_P \sigma_t^\vphi(y) w_t^*
\end{equation}
for all $x,y \in M$ and $t \in \R$.

Define the positive element $X \in q \langle M,e_P \rangle q$ by
$$X = \sum_{j = 1}^k x_j e_P x_j^* \; .$$
Also define the normal positive functional $\Om \in \langle M,e_P \rangle_*$ given by
$$\Om(T) = \sum_{i=1}^k \vphih(e_P x_i^* T x_i e_P)$$
for all $T \in \langle M,e_P \rangle$. Using \eqref{eq.formula-sigma-psih} and \eqref{eq.what-we-get}, we get for every $t \in \R$,
\begin{align*}
\Om(\sigma^{\psih}_t(X)) &= \sum_{i,j=1}^k \vphih\bigl( e_P x_i^* w_t \sigma_t^\vphi(x_j) e_P \sigma_t^\vphi(x_j)^* w_t^* x_i e_P \bigr) \\
&= \sum_{i,j=1}^k \vphi\bigl( \rE_P(x_i^* w_t \sigma_t^\vphi(x_j)) \, \rE_P(x_i^* w_t \sigma_t^\vphi(x_j))^* \bigr) \geq \delta \; .
\end{align*}
Define $K$ as the $\si$-weakly closed convex hull of $\{\sigma_t^{\psih}(X) : t \in \R \}$ inside $q \langle M ,e_P \rangle q$. Note that $\|Y\| \leq \|X\|$ for all $Y \in K$. Also, every $Y \in K$ is positive and satisfies $\psih(Y) \leq \psih(X) < +\infty$, by the $\sigma_t^{\psih}$-invariance and $\sigma$-weak lower semicontinuity of $\psih$. We then also have
$$\psih(Y^* Y) = \psih(Y^2) \leq \|Y\| \, \psih(Y) \leq \|X\| \, \psih(X)$$
for all $Y \in K$. By \cite[Lemma 4.4]{HI15}, the image of $K$ in $\rL^2(\langle M,e_P \rangle,\psih)$ is norm closed. So, there is a unique element $X_0 \in K$ where the function $Y \mapsto \psih(Y^*Y)$ attains its minimal value. Since this function is $\sigma_t^{\psih}$-invariant, it follows that $\sigma_t^{\psih}(X_0) = X_0$ for all $T$. Since $\Om(\sigma_t^{\psih}(X)) \geq \delta$ for all $t \in \R$, also $\Om(X_0) \geq \delta$, so that $X_0 \neq 0$. Since $\rT_M \circ \sigma_t^{\psih} = \sigma_t^\psi \circ \rT_M$ and since $\rT_M$ is $\sigma$-weakly lower semicontinuous, we get that $\|\rT_M(Y)\| \leq \|\rT_M(X)\| < +\infty$ for all $Y \in K$. In particular, $\|\rT_M(X_0)\| < +\infty$.

Take $\eps > 0$ small enough such that the spectral projection $e = 1_{[\eps,+\infty)}(X_0)$ is nonzero. It follows that $e$ is a projection in $q \langle M,e_P \rangle q$ satisfying $\sigma_t^{\psih}(e) = e$ for all $t \in \R$ and $\|\rT_M(e)\| < +\infty$. By Lemma \ref{lem.technical} below, we may assume that $e \prec e_P$ inside $\langle M,e_P \rangle$. Take a partial isometry $V \in \langle M,e_P \rangle$ such that $V^* V = e$ and $V V^* \leq e_P$. Let $p_0 \in P$ be the unique projection such that $VV^* = p_0 e_P$. We get that $V = p_0 V$. Since $e \leq q$, we also have $V = V q$.

Since $\|\rT_M(V^* V)\| = \|\rT_M(e)\| < +\infty$, it follows from the push down lemma \cite[Proposition 2.2]{ILP96} (where the factoriality assumption is unnecessary) that
$$V = e_P V = e_P \, \rT_M(e_P V) = e_P \, \rT_M(V) \; .$$
Write $v = \rT_M(V)$. Then, $v \in M$ and $V = e_P v$. By construction, $v \in p_0 M q$.

Since $e_P \langle M , e_P \rangle e_P = P e_P$, we can uniquely define $u_t \in P$ such that
$$u_t e_P = V w_t \sigma_t^{\vphih}(V^*)$$
for all $t \in \R$. Since $V^* V = e$ and $w_t \si_t^{\vphih}(V^*V) w_t^* = \sigma_t^{\psih}(e) = e$, we get that $u_t u_t^* = p_0$ and $u_t^* u_t = \sigma_t^\vphi(p_0)$ for all $t \in \R$. Also, $t \mapsto u_t$ is strongly continuous and
\begin{align*}
u_t \sigma_t^\vphi(u_s) \, e_P & = u_t e_P \, \sigma_t^{\vphih}(u_s e_P) = V w_t \sigma_t^{\vphih}(V^*) \; \sigma_t^{\vphih}(V w_s \sigma_s^{\vphih}(V^*)) \\
& = V w_t \, \sigma_t^{\vphih}(V^* V) \, \sigma_t^\vphi(w_s) \, \sigma_{t+s}^{\vphih}(V^*) = V \, \sigma_t^{\psih}(e) \, w_t \sigma_t^\vphi(w_s) \, \sigma_{t+s}^{\vphih}(V^*) \\
& = V \, w_{t+s} \, \sigma_{t+s}^{\vphih}(V^*) = u_{t+s} e_P
\end{align*}
for all $s,t \in \R$. So, $(u_t)_{t \in \R}$ is a $1$-cocycle for $\vphi|_P$. By \cite[Th\'{e}or\`{e}me 1.2.4]{Co72} (see also \cite[Theorem VIII.3.21]{Ta03} for a formulation adapted to non faithful states), there is a unique faithful normal semifinite weight $\theta$ on $p_0 P p_0$ such that $[\rD \theta : \rD \vphi|_P]_t = u_t$ for all $t \in \R$. Define the faithful normal semifinite weight $\theta_1$ on $p_0 M p_0$ by $\theta_1 = \theta \circ \rE_P$. By \cite[Theorem 4.7]{Ha77}, we have that $[\rD \theta_1 : \rD \vphi]_t = u_t$ for all $t \in \R$.

Since $u_t \in P$, we get that
\begin{align*}
e_P u_t \sigma_t^\vphi(v) & = u_t e_P \sigma_t^\vphi(v) = u_t e_P \sigma_t^{\vphih}(V) \\
&= V w_t \sigma_t^{\vphih}(V^* V) = V \, \sigma_t^{\psih}(V^*V) \, w_t = V w_t = e_P v w_t \; .
\end{align*}
Applying $\rT_M$, we conclude that $u_t \sigma_t^\vphi(v) = v w_t$ for all $t \in \R$. Replacing $v$ by its polar part, we may assume that $v \in M$ is a partial isometry such that $p_1 = vv^* \in (p_0 M p_0)^{\theta_1}$ and $q_1 = v^* v \in q M^\psi q$.
We then get that
$$[\rD(v^* \theta_1 v) : \rD\vphi]_t = v^* [\rD\theta_1 : \rD\vphi]_t \sigma_t^\vphi(v) = v^* u_t \sigma_t^\vphi(v) = q_1 [\rD \psi : \rD\vphi]_t = [\rD(\psi q_1) : \rD\vphi]_t$$
for all $t \in \R$. We conclude that $\psi q_1 = v^* \theta_1 v$.

In particular, $\theta(\rE_P(p_1)) = \theta_1(p_1) = \psi(q_1) < +\infty$. Also, $\sigma_t^\theta(\rE_P(p_1)) = \rE_P(\sigma_t^{\theta_1}(p_1)) = \rE_P(p_1)$, for all $t \in \R$. We can thus find a nonzero spectral projection $f$ of $\rE_P(p_1)$ such that $f v \neq 0$, $\theta(f) < +\infty$ and $f \in (p_0 P p_0)^\theta$. Since $u_t \sigma_t^\vphi(v) = v w_t$, $[\rD \theta : \rD \vphi|_P]_t = u_t$ and $f \in (p_0 P p_0)^\theta$, we get that
$$u_t \sigma_t^\vphi(f v) = f u_t \sigma_t^\vphi(v) = f v w_t \; .$$
We then replace $\theta$ by $\theta f$ and $v$ by the polar part of $f v$. Then, $\theta$ is a faithful normal positive functional on $fPf$, the projection $p = vv^*$ belongs to $(f M f)^{\theta \circ \rE_P}$, the projection $q_0 = v^* v$ belongs to $q M^\psi q$ and $\psi q_0 = v^* (\theta \circ \rE_P) v$. Adding to $\theta$ an arbitrary faithful normal state on $(1-f)P(1-f)$, it follows that $({\rm ii})$ holds.

Conversely, assume that $({\rm ii})$ holds. Take $\theta$, $q_0$, $p$ and $v$ as in the statement of $(\rm ii)$. Define $w_t = [\rD \psi : \rD\vphi]_t$ and $u_t = [\rD \theta : \rD \vphi|_P]_t$. Since $\psi q_0 = v^* (\theta \circ \rE_P) v$, we get that $v w_t = u_t \sigma_t^\vphi(v)$ for all $t \in \R$. This means that
$$v \; \Pi_{\vphi,\psi}(\lambda_\psi(t)) = \Pi_{\vphi|_P,\theta}(\lambda_\theta(t)) \; v$$
for all $t \in \R$. Also, $v \; \Pi_{\vphi,\psi}(q r) = v \; \Pi_{\vphi,\psi}(r) \neq 0$ for every nonzero finite trace projection $r \in \rL_\psi(\R)$. We conclude that
$$\Pi_{\vphi,\psi}(\rL_\psi(\R) q r) \prec_{\core_\vphi(M)} \core_\vphi(P)$$
for every nonzero finite trace projection $r \in \rL_\psi(\R)$, so that $({\rm i})$ holds.
\end{proof}

Applying Theorem \ref{thm-key} to the case $P = \C 1$, we get the following result.

\begin{cor}\label{cor.key}
Let $M$ be a von Neumann algebra with faithful normal states $\psi,\vphi \in M_*$ and let $q \in M^\psi$ be a nonzero projection. Then the following statements are equivalent.
\begin{itemize}
\item [$({\rm i})$] There exists a nonzero finite trace projection $r \in \rL_\psi(\R)$ such that $$\Pi_{\vphi,\psi}(\rL_\psi(\R) q r) \prec_{\core_\vphi(M)} \rL_\vphi(\R) \; .$$
\item [$({\rm ii})$] There exists a nonzero partial isometry $v \in M$ such that $p = vv^*$ belongs to $M^\vphi$, $q_0 = v^* v$ belongs to $q M^\psi q$, and
$$\frac{1}{\psi(q_0)} \, \psi q_0 = \frac{1}{\vphi(p)} \, v^* \vphi v \; .$$
\end{itemize}
\end{cor}

\begin{lem}\label{lem.technical}
Let $\psi$ be a faithful normal semifinite weight on a von Neumann algebra $N$ and $e \in N^\psi$ a projection satisfying $0 < \psi(e) < +\infty$. Let $e_1 \in N$ be any projection with central support equal to $1$. Then there exists a nonzero projection $e_0 \in e N^\psi e$ satisfying $e_0 \prec e_1$ inside $N$.
\end{lem}
\begin{proof}
Since the central support of $e_1$ equals $1$, we can find a nonzero projection $f \in N$ such that $f \leq e$ and $f \prec e_1$. Define the faithful normal state $\theta$ on $eNe$ given by $\theta(x) = \psi(e)^{-1} \psi(x)$ for all $x \in eNe$. By \cite[Lemma 2.1]{HU15}, there exists a projection $e_0 \in (eNe)^\theta$ such that $e_0 \sim f$ inside $eNe$. Then, $e_0$ is a nonzero projection in $e N^\psi e$ and $e_0 \prec e_1$ inside $N$.
\end{proof}

\section{Isomorphisms of free Araki--Woods factors}

The isomorphism part of Corollary \ref{corB} follows from the following result that we deduce from \cite{Sh96}.

\begin{thm}\label{thm.isom}
Let $\mu$ be any finite symmetric Borel measure on $\R$ and $m : \R \to \N \cup \{+\infty\}$ any symmetric Borel multiplicity function. Denote by $\Lambda$ the subgroup of $\R$ generated by the atoms of $\mu$ and assume that $\Lambda \neq \{0\}$. There is an isomorphism
$$\Gamma(\mu,m)\dpr \cong \Gamma(\mu \ast \delta_{\Lambda},+\infty)\dpr$$
preserving the free quasi-free states, where $\delta_{\Lambda}$ denotes any atomic finite symmetric Borel measure on $\R$ with set of atoms equal to $\Lambda$.
\end{thm}
\begin{proof}
For every $0 < a < 1$, we denote by $B_a$ the von Neumann algebra $\rB(\ell^2(\N))$ equipped with the faithful normal state $\theta_a$ given by
$$\theta_a(T) = (1-a) \sum_{k=0}^\infty a^k \langle T(\delta_k),\delta_k \rangle \; ,$$
where $(\delta_k)_{k \in \N}$ is the standard orthonormal basis of $\ell^2(\N)$.
Throughout the proof of the theorem, we always assume that a free Araki--Woods factor comes with its free quasi-free state and that all free products are taken w.r.t.\ the canonical states that we fixed. We always equip a free product with the free product state and a tensor product with the tensor product state.

We first prove that for every $0 < a < 1$, for all finite symmetric Borel measures $\mu$ on $\R$ and all symmetric Borel multiplicity functions $m : \R \to \N \cup \{+\infty\}$, there exists a state preserving isomorphism
\begin{equation}\label{eq.basic-iso}
\Gamma(\mu,m)\dpr \ast B_a \cong \Gamma(\mu \ast \delta_{\Z \log(a)},+\infty)\dpr \ovt B_a \; .
\end{equation}
To prove \eqref{eq.basic-iso}, fix an orthogonal representation $(U_t)_{t \in \R}$ of $\R$ on a real Hilbert space $H_\R$ having $(\mu,m)$ as its spectral invariant. Denote by $H = H_\R + {\rm i} H_\R$ the complexification of $H_\R$. Define the positive operator $\Delta$ on $H$ such that $U_t = \Delta^{\ri t}$ and denote by $J : H \recht H$ the anti-unitary operator given by $J(\xi + {\rm i} \eta) = \xi - {\rm i} \eta$ for all $\xi,\eta \in H_\R$. Define $H_1 = H \ot \ell^2(\N^2)$. On $H_1$, we consider the positive operator $\Delta_1$ and anti-unitary operator $J_1$ given by
$$\Delta_1 (\xi \ot \delta_{ij}) = a^{j-i} \, \Delta(\xi) \ot \delta_{ij} \quad\text{and}\quad J_1(\xi \ot \delta_{ij}) = J(\xi) \ot \delta_{ji}$$
for all $i,j \in \N$ and $\xi \in D(\Delta)$. Here, $(\delta_{ij})$ denotes the standard orthonormal basis of $\ell^2(\N^2)$. Note that $J_1 \Delta_1 J_1 = \Delta_1^{-1}$.

Denote by $\cF(H_1)$ the full Fock space of $H_1$ and by $\theta_1$ the vector state on $\rB(\cF(H_1))$ implemented by the vacuum vector. For every $\xi \in H$, define the element $L(\xi) \in \rB(\cF(H_1)) \ovt B_a$ given by
$$L(\xi) = \sum_{i,j=0}^\infty \ell(\xi \ot \delta_{ij}) \ot e_{ij} \, \sqrt{(1-a) a^{i}} \; .$$
By \cite[Theorem 5.2]{Sh96}, we can realize $\Gamma(\mu,m)\dpr \ast B_a$ as the von Neumann algebra $\cM$ generated by
$$\{ L(\xi) + L(J\Delta^{1/2} \xi)^* : \xi \in D(\Delta^{1/2})\} \quad\text{and}\quad 1 \ot B_a \; .$$
Moreover, the free product state on $\Gamma(\mu,m)\dpr \ast B_a$ is given by the restriction of $\theta_1 \ot \theta_a$ to $\cM$.

To conclude the proof of \eqref{eq.basic-iso}, it thus suffices to show that there is a state preserving isomorphism
\begin{equation}\label{eq.enough}
\bigl((1 \ot e_{00})\cM(1 \ot e_{00}),\theta_1) \cong \bigl( \Gamma(\mu \ast \delta_{\Z \log(a)},+\infty)\dpr , \vphi_{\mu \ast \delta_{\Z \log(a)},+\infty} \bigr) \; .
\end{equation}
The left hand side of \eqref{eq.enough} is generated by the operators
$$(1 \ot e_{0i}) \, (L(\xi) + L(J\Delta^{1/2} \xi)^*) \, (1 \ot e_{j0}) = \sqrt{(1-a)a^i} \; \Bigl( \ell\bigl(\xi \ot \delta_{ij}\bigr) + \ell\bigl(J_1 \Delta_1^{1/2} (\xi \ot \delta_{ij})\bigr)^* \Bigr) \; .$$
So, the left hand side of \eqref{eq.enough} equals $(\Gamma(\mu_1,m_1)\dpr,\vphi_{\mu_1,m_1})$ where $\mu_1$ and $m_1$ are chosen so that the measure class and multiplicity function of $\log(\Delta_1)$ equal $\cC(\mu_1)$ and $m_1$. One checks that $\cC(\mu_1) = \cC(\mu \ast \delta_{\Z \log(a)})$ and $m_1 = +\infty$ a.e. So we have proved the existence of the state preserving isomorphism \eqref{eq.basic-iso}.

We next prove that for every $t > 0$, there is a state preserving isomorphism
\begin{equation}\label{eq.next-iso}
\Gamma(\mu,m)\dpr \ast \Gamma(\delta_t + \delta_{-t},1)\dpr \cong \Gamma(\mu \ast \delta_{\Z t} \vee \delta_{\Z t} , +\infty)\dpr \; .
\end{equation}
By \cite[Theorem 4.8]{Sh96}, there is a state preserving isomorphism
$$\Gamma(\delta_t + \delta_{-t},1)\dpr \cong \rL^\infty([0,1]) \ast B_{\exp(-t)} \; .$$
By \cite[Theorem 2.11]{Sh96}, the free Araki--Woods functor $\Gamma$ turns direct sums into free products. Writing $\mu_1 = \mu + \delta_0$, $m_1 = m + \delta_0$ and using \eqref{eq.basic-iso}, we obtain the state preserving isomorphisms
$$\Gamma(\mu,m)\dpr \ast \Gamma(\delta_t + \delta_{-t},1)\dpr \cong \Gamma(\mu_1,m_1)\dpr \ast B_{\exp(-t)} \cong \Gamma(\mu_1 \ast \delta_{\Z t},+\infty)\dpr \ovt B_{\exp(-t)} \; .$$
Applying this to $(\mu,m) = (\mu_1 \ast \delta_{\Z t},+\infty)$, we also have the state preserving isomorphisms
$$\Gamma(\mu_1 \ast \delta_{\Z t},+\infty)\dpr \cong \Gamma(\mu_1 \ast \delta_{\Z t},+\infty)\dpr \ast \Gamma(\delta_t + \delta_{-t},1)\dpr \cong \Gamma(\mu_1 \ast \delta_{\Z t},+\infty)\dpr \ovt B_{\exp(-t)} \; .$$
Combining both, it follows that \eqref{eq.next-iso} holds.

We are now ready to prove the theorem. Fix an atom $t>0$ of $\mu$. Writing $(\mu,m) = (\mu_0,m_0) + (\delta_{t} + \delta_{-t},1)$, we get from \eqref{eq.next-iso} the state preserving isomorphism
\begin{equation}\label{eq.final}
\Gamma(\mu,m)\dpr \cong \Gamma(\mu_0,m_0)\dpr \ast \Gamma(\delta_t + \delta_{-t},1)\dpr \cong \Gamma(\mu_0 \ast \delta_{\Z t} \vee \delta_{\Z t},+\infty)\dpr \cong \Gamma(\mu \ast \delta_{\Z t},+\infty)\dpr \; .
\end{equation}
Let $\{t_n : n \geq 0\}$ be the positive atoms of $\mu$, with repetitions if there are only finitely many of them. For every $n \geq 0$, define
$$\mu_n = \mu \ast \delta_{\Z t_0} \ast \cdots \ast \delta_{\Z t_n} \; .$$
For every $n$, $t_{n+1}$ is an atom of $\mu_n$. Repeatedly applying \eqref{eq.final}, we find the state preserving isomorphisms
$$\Gamma(\mu,m)\dpr \cong \Gamma(\mu_0,+\infty)\dpr \cong \Gamma(\mu_n,+\infty)\dpr \; .$$
So, we also get state preserving isomorphisms
$$
\Gamma(\mu,m)\dpr \cong \Gamma(\mu_0,+\infty)\dpr \cong \bigast_{n \in \N} \Gamma(\mu_0,+\infty)\dpr
\cong \bigast_{n \in \N} \Gamma(\mu_n,+\infty)\dpr \cong \Gamma(\vee_{n \in \N} \mu_n, +\infty) \; .
$$
Since $\vee_{n \in \N} \mu_n$ is equivalent with $\mu \ast \delta_\Lambda$, the theorem follows.
\end{proof}

We deduce the isomorphism part of Theorem \ref{thmF} from the following result, generalizing \cite[Theorem 5.1]{Sh97a} and proved using the same methods. For every faithful normal state $\psi$ on a von Neumann algebra $A$ and for every nonzero projection $p \in A$, we denote by $\psi_p$ the faithful normal state on $pAp$ given by $\psi_p(a) = \psi(p)^{-1} \psi(a)$ for all $a \in pAp$.

\begin{prop}\label{prop.corner-free-product-aw}
Let $\mu$ be a finite symmetric Borel measure on $\R$ and fix the free Araki--Woods factor $(M,\vphi) = (\Gamma(\mu,+\infty)\dpr , \vphi_{\mu,+\infty})$. Let $A$ be a von Neumann algebra with a faithful normal state $\psi$ having a factorial centralizer $A^\psi$. For every nonzero projection $p \in A^\psi$, there is a state preserving isomorphism
$$\Bigl(p\bigl( (M,\vphi) \ast (A,\psi) \bigr) p , (\vphi \ast \psi)_p\Bigr) \cong \Bigl((M,\vphi) \ast (pAp,\psi_p), \vphi \ast \psi_p\Bigr) \; .$$
\end{prop}

To prove Proposition \ref{prop.corner-free-product-aw}, we need the following lemma. It is a direct consequence of \cite[Corollary 2.5]{Sh97a}. To formulate the lemma, we use yet another convention for the construction of free Araki--Woods factors. We call \emph{involution} on a Hilbert space $H$ any closed densely defined antilinear operator $S$ satisfying $S(\xi) \in \rD(S)$ and $S(S(\xi)) = \xi$ for all $\xi \in \rD(S)$. Taking the polar decomposition $S = J \Delta^{1/2}$ of such an involution, we obtain an anti-unitary operator $J$ and a nonsingular positive selfadjoint operator $\Delta$ satisfying $J \Delta J = \Delta^{-1}$. Denoting by $U_t$ the restriction of $\Delta^{\ri t}$ to the real Hilbert space $H_\R = \{\xi \in H : J(\xi) = \xi\}$, we obtain an orthogonal representation $(U_t)_{t \in \R}$. Every orthogonal representation of $\R$ arises in this way. The associated free Araki--Woods factor can be realized on the full Fock space $\cF(H)$ as the von Neumann algebra generated by the operators $\ell(\xi) + \ell(S(\xi))^*$, $\xi \in \rD(S)$. We denote this realization of the free Araki--Woods factor as $\Gamma(H,S)\dpr$.

\begin{lem}\label{lem.model}
Let $K$ be a Hilbert space and $\Omega \in K$ a unit vector. Let $H$ be a Hilbert space and $H_0 \subset H$ a total subset. Assume that
\begin{itemize}
\item $A \subset \rB(K)$ is a von Neumann subalgebra and $\langle \, \cdot \, \Om , \Om \rangle$ defines a faithful state $\psi$ on $A$,
\item for every $\xi \in H_0$, we are given an operator $L(\xi) \in \rB(K)$,
\end{itemize}
such that the following conditions hold:
\begin{itemize}
\item [$(\rm i)$] $L(\xi_1)^* a L(\xi_2) = \psi(a) \, \langle \xi_2,\xi_1 \rangle \, 1$ for all $\xi_1,\xi_2 \in H_0$ and $a \in A$,
\item [$(\rm ii)$] $L(\xi)^* a \Om = 0$ for all $\xi \in H_0$ and $a \in A$,
\item [$(\rm iii)$] denoting by $\cA$ the $*$-algebra generated by $A$ and $\{L(\xi) : \xi \in H_0\}$, we have that $\cA \Om$ is dense in $K$.
\end{itemize}
Then, $L$ can be uniquely extended to a linear map $L : H \recht \rB(K)$ such that the above properties remain valid. For every involution $S$ on $H$ with associated free Araki--Woods factor $\Gamma(H,S)\dpr$, there is a unique normal homomorphism
$$\pi : \bigl( \Gamma(H,S)\dpr, \vphi_{(H,S)} \bigr) \ast (A,\psi) \recht \rB(K)$$
satisfying $\pi\bigl(\ell(\xi) + \ell(S(\xi))^*\bigr) = L(\xi) + L(S(\xi))^*$ for all $\xi \in \rD(S)$ and $\pi(a) = a$ for all $a \in A$. Also, $\langle \pi(\,\cdot\,)\Om,\Om \rangle$ equals the free product state $\vphi_{(H,S)} \ast \psi$.
\end{lem}

Using Lemma \ref{lem.model}, we can prove Proposition \ref{prop.corner-free-product-aw}.

\begin{proof}[Proof of Proposition \ref{prop.corner-free-product-aw}]
Since $A^\psi$ is a factor, we can choose partial isometries $v_i \in A^\psi$, $i \geq 1$, such that $v_i^* v_i \leq p$, $\sum_{i=1}^\infty v_i v_i^* = 1$ and $\psi(v_i^* v_i) = \psi(p)/n_i$ for some integers $n_i \geq 1$. We can then also choose partial isometries $w_{is} \in p A^\psi p$, $s = 1,\ldots, n_i$, such that $w_{is} w_{is}^* = v_i^* v_i$ for all $s$ and $\sum_{s=1}^{n_i} w_{is}^* w_{is} = p$.

Since $(M,\vphi)$ is a free Araki--Woods factor with infinite multiplicity, we can choose an involution $S_0$ on a Hilbert space $H_0$ and realize $(M,\vphi)$ as $\Gamma(H,S)\dpr$, where $H=H_0 \ot \ell^2(\N^2)$ and $S$ is given by $S(\xi \ot \delta_{kl}) = S_0(\xi) \ot \delta_{lk}$ for all $\xi \in \rD(S_0)$ and all $k,l \geq 1$. We then consider the standard free product representation for $\Gamma(H,S)\dpr \ast A$ on the Hilbert space $K$ with vacuum vector $\Om$. Note that $p \bigl( \Gamma(H ,S)\dpr \ast A \bigr) p$ is generated by
\begin{equation}\label{eq.generators}
p A p \cup \Bigl\{ v_i^* \bigl( \ell(\xi \ot \delta_{kl}) + \ell(S_0(\xi) \ot \delta_{lk})^* \bigr) v_j \Bigm| i,j,k,l \geq 1 , \xi \in \rD(S_0) \Bigr\} \; .
\end{equation}
For all $k,l \geq 0$, $i,j \geq 1$ and $\xi \in H$, define
$$L_{ijkl}(\xi) = \psi(p)^{-1/2} \sum_{s=1}^{n_i} \sum_{t=1}^{n_j} w_{is}^* v_i^* \, \ell(\xi \ot \delta_{n_i k + s , n_j l + t}) \, v_j w_{jt} \; .$$
A direct computation shows that
$$L_{i'j'k'l'}(\xi')^* \, a \, L_{ijkl}(\xi) = \delta_{ijkl,i'j'k'l'} \, \langle \xi,\xi' \rangle \, \psi_p(a) \, p$$
for all $i,j,i',j' \geq 1$, $k,l,k',l' \geq 0$, $\xi,\xi' \in H$ and $a \in pAp$.

Applying Lemma \ref{lem.model} to the Hilbert space $H_1 = H \ot \ell^2(\N^2 \times \N_0^2)$ with involution $S_1(\xi \ot \delta_{ijkl}) = S_0(\xi) \ot \delta_{jilk}$, it follows that $\Gamma(H_1,S_1)\dpr \ast pAp$ can be realized as the von Neumann algebra $N$ generated by
$$p A p \cup \Bigl\{ L_{ijkl}(\xi) + L_{jilk}(S_0(\xi))^* \Bigm| i,j \geq 1, k,l\geq 0, \xi \in \rD(S_0) \Bigr\} \; ,$$
with the free product state being implemented by $\psi(p)^{-1/2} p \Om$.

Note that
$$w_{is} \; \bigl( L_{ijkl}(\xi) + L_{jilk}(S_0(\xi))^* \bigr) \; w_{jt}^* = \psi(p)^{-1/2} \; v_i^* \; \bigl( \ell(\xi \ot \delta_{n_i k + s, n_j l + t}) + \ell(S_0(\xi) \ot \delta_{n_j l + t, n_i k + s})^* \bigr) \; v_j \; .$$
For fixed $i,j \geq 1$, the parameters $n_i k + s$ and $n_j l + t$ with $k,l \geq 0$, $s=1,\ldots,n_i$ and $t=1,\ldots,n_j$ exactly run through $\N^2$. So, we find back the generating set of \eqref{eq.generators} and conclude that $p\bigl( \Gamma(H ,S)\dpr \ast A \bigr) p$ equals $\Gamma(H_1,S_1)\dpr \ast pAp$ in a state preserving way. Since also $\Gamma(H_1,S_1)\dpr \cong (M,\vphi)$ in a state preserving way, this concludes the proof of the proposition.
\end{proof}

\section{Proofs of Theorem \ref{thmA} and Corollaries \ref{corB}, \ref{corC}, \ref{corD}}

Combining Corollary \ref{cor.key} with the deformation/rigidity theorems for free Araki--Woods factors and for free product factors obtained in \cite{HR10,HU15}, we get the following theorem.

\begin{thm}\label{thm.meta}
Let $(M,\vphi)$ be either a free Araki--Woods factor with its free quasi-free state or a free product $\ast_n (M_n,\vphi_n)$ of amenable von Neumann algebras equipped with the free product state. Let $\psi \in M_*$ be any faithful normal state on $M$ and denote by $[\rD \psi : \rD \vphi]_t$ Connes' Radon--Nikodym $1$-cocycle between $\psi$ and $\vphi$. Let $z \in \cZ(M^\psi)$ be the central projection such that $M^\psi (1-z)$ is amenable and $M^\psi z$ has no amenable direct summand.

There exists a sequence of partial isometries $v_n \in M$ such that the projection $q_n = v_n v_n^*$ belongs to $M^\psi$, the projection $p_n = v_n^* v_n$ belongs to $M^\vphi$, $\sum_n q_n = z$,
$$q_n \, [\rD\psi : \rD\vphi]_t = \lambda_n^{\ri t} \, v_n \, \sigma_t^\vphi(v_n^*) \quad \text{and} \quad \psi q_n = \lambda_n \, v_n \vphi v_n^* \; ,$$
with $\lambda_n = \psi(q_n)/\vphi(p_n)$.
\end{thm}
\begin{proof}
Let $q \in M^\psi$ be a nonzero projection such that $q M^\psi q$ has no amenable direct summand. Let $r_0 \in \rL_\psi(\R)$ be a nonzero finite trace projection. Put $r = \Pi_{\vphi,\psi}(q r_0)$. Then $r$ is a nonzero finite trace projection in the core $c_\vphi(M)$ and $\Pi_{\vphi,\psi}(\rL_\psi(\R)) r$ commutes with $q M^\psi q r$. Since $q M^\psi q r$ has no amenable direct summand, it follows from \cite[Theorem 5.2]{HR10} (in the case where $M$ is a free Araki--Woods factor) and \cite[Theorem 4.3]{HU15} (in the case where $M$ is a free product of amenable von Neumann algebras) that $\Pi_{\vphi,\psi}(\rL_\psi(\R)) r \prec_{c_\vphi(M)} \rL_\vphi(\R)$.

By Theorem \ref{thm-key}, we find a nonzero partial isometry $v \in q M$ such that the projection $q_0 = vv^*$ belongs to $M^\psi$, the projection $p = v^* v$ belongs to $M^\vphi$ and $\psi q = \lambda \, v \vphi v^*$. In particular, $\lambda = \psi(q_0)/\vphi(p)$ and $q_0 \, [\rD \psi : \rD \vphi]_t = \lambda^{\ri t} \, v \sigma_t^\vphi(v^*)$.

Since $q \in M^\psi$ was an arbitrary nonzero projection such that $q M^\psi q$ has no amenable direct summand, the theorem follows by a maximality argument.
\end{proof}

In order to apply Theorem \ref{thm.meta} to the classification of free Araki--Woods factors, we need the following description of the centralizer of the free quasi-free state.

\begin{remark}\label{rem.centralizer}
When $M = \Gamma(\mu,m)^{\prime\prime}$ is an arbitrary free Araki--Woods factor with free quasi-free state $\vphi = \vphi_{\mu,m}$, the centralizer $M^\vphi$ can be described as follows. Denote by $M_a = \Gamma(\mu_a,m)^{\prime\prime}$ the almost periodic part of $M$. First note that $M^\vphi = M_a^\vphi$. So if $\mu_a = 0$, we have $M^\vphi = \C 1$. If $\mu_a$ is concentrated on $\{0\}$, we conclude that $M^\vphi = M_a = \rL(\F_{m(0)})$, where the last isomorphism follows because the free Araki--Woods factor associated with the $m$-dimensional trivial representation, i.e.\ $\Gamma(\delta_0,m)\dpr$, is isomorphic with $\rL(\F_m)$.  When $\mu_a(\log \lambda) > 0$ for some $0 < \lambda < 1$, there is a state preserving inclusion $T_\lambda \subset M_a$, where $T_\lambda$ is the unique free Araki--Woods factor of type III$_\lambda$ (see \cite[Section 4]{Sh96}). It then follows from \cite[Corollary 6.8]{Sh96} that $M^\vphi$ is a factor that contains a copy of the free group factor $\rL(\F_\infty)$, so that $M^\vphi$ is nonamenable. Actually, using \cite{Dy96}, we get that $M^\vphi \cong \rL(\F_\infty)$ in this case.
\end{remark}

Theorem \ref{thmA} is a particular case of the following more general result.

\begin{thm}
Let $\mu, \nu$ be finite symmetric Borel measures on $\R$ and $m, n : \R \to \N \cup \{+\infty\}$ symmetric Borel multiplicity functions. Assume that $\nu_a \neq 0$ and either $\supp(\nu_a) \neq \{0\}$ or $\supp(\nu_a) = \{0\}$ with $n(0) \geq 2$.

If the free Araki--Woods factors $\Gamma(\mu, m)\dpr$ and $\Gamma(\nu, n)\dpr$ are isomorphic then there exist nonzero projections $p \in (\Gamma(\mu, m)\dpr)^{\varphi_{\mu, m}}$ and $q \in (\Gamma(\nu, n)\dpr)^{\varphi_{\nu, n}}$ and a state preserving isomorphism
$$\left(p \, \Gamma(\mu, m)\dpr \, p, (\varphi_{\mu, m})_p \right) \cong \left(q \, \Gamma(\nu, n)\dpr \, q, (\varphi_{\nu, n})_q \right)$$
where $(\varphi_{\mu, m})_p = \frac{\varphi_{\mu, m}(p \, \cdot \, p)}{\varphi_{\mu, m}(p)}$ and $(\varphi_{\nu, n})_q = \frac{\varphi_{\nu, n}(q \, \cdot \, q)}{\varphi_{\nu, n}(q)}$.

In particular, the joint measure classes $\mathcal C(\bigvee_{k \geq 1} \mu^{\ast k})$ and $\mathcal C(\bigvee_{k \geq 1} \nu^{\ast k})$ are equal.

Moreover, in case $\supp(\nu_a) \neq \{0\}$ or $\supp(\nu_a) = \{0\}$ with $n(0)=+\infty$, there exists a state preserving isomorphism $(\Gamma(\mu,m)\dpr,\vphi_{\mu,m}) \cong (\Gamma(\nu,n)\dpr,\vphi_{\nu,n})$.
\end{thm}

\begin{proof}
Put $(M,\vphi) := (\Gamma(\mu, m)\dpr,\vphi_{\mu,m})$ and $(N,\theta) := (\Gamma(\nu, n)\dpr,\vphi_{\nu,n})$. Let $\pi : M \to N$ be any isomorphism between $M$ and $N$. Put $\psi := \theta \circ \pi$. By our assumptions on $\nu$ and Remark \ref{rem.centralizer}, the centralizer $M^\psi$ is nonamenable. By Theorem \ref{thm.meta}, we find a nonzero partial isometry $v \in M$ such that $p = v^*v \in M^\varphi$, $q = vv^* \in M^\psi$ and $\Ad(v) : (pMp, \varphi_p) \to (qMq, \psi_q)$ is state preserving. It follows in particular that $p M^\varphi p = (pMp)^{\varphi_p} \cong (qMq)^{\psi_q} = qM^\psi q$ is a nonamenable ${\rm II_1}$ factor. So $M^\vphi$ cannot be abelian and Remark \ref{rem.centralizer} implies that $M^\vphi$ is a ${\rm II_1}$ factor. Applying Lemma \ref{lem-measure} twice, we have
$$\mathcal C(\bigvee_{k \in \N} \mu^{\ast k}) = \mathcal C(\Delta_\varphi) = \mathcal C(\Delta_{\varphi_p}) = \mathcal C(\Delta_{\psi_q}) = \mathcal C(\Delta_\psi) = \mathcal C(\bigvee_{k \in \N} \nu^{\ast k}).$$
This implies that $\mathcal C(\bigvee_{k \geq 1} \mu^{\ast k}) = \mathcal C(\bigvee_{k \geq 1} \nu^{\ast k})$.

Assume now that either $\supp(\nu_a) \neq \{0\}$ or $\supp(\nu_a) = \{0\}$ with $n(0)=+\infty$. In the latter case where $\nu(\{0\}) > 0$ and $n(0) = +\infty$, we use that the free Araki--Woods functor $\Gamma$ turns direct sums into free products (see \cite[Theorem 2.11]{Sh96}) and conclude that there exists a state preserving isomorphism
\begin{equation}\label{eq.stable-infty}
(N,\theta) \cong (N,\theta) \ast (\rL(\F_\infty),\tau) \; .
\end{equation}
In the case where $\nu$ has at least one atom different from $0$, it follows similarly from the classification of almost periodic free Araki--Woods factors (see \cite{Sh96}) that \eqref{eq.stable-infty} holds.

Put $q_0 = \pi(q)$. Above, we have proved that there exists a state preserving isomorphism $(pMp,\vphi_p) \cong (q_0 N q_0 , \theta_{q_0})$. Taking a smaller $p$ if needed, we may assume that $\vphi(p) = 1/k$ for some integer $k \geq 1$. Combining \eqref{eq.stable-infty} with Proposition \ref{prop.corner-free-product-aw} and the fact that the fundamental group of $\rL(\F_\infty)$ equals $\R^\ast_+$ (see \cite{Ra91}), it follows that there exists a state preserving isomorphism $(q_0 N q_0,\theta_{q_0}) \cong (q_1 N q_1,\theta_{q_1})$ whenever $q_1 \in N^\theta$ is a nonzero projection.

Choose a projection $q_1 \in N^\theta$ with $\theta(q_1) = 1/k$. So, there exists a state preserving isomorphism $(pMp,\vphi_p) \cong (q_1 N q_1 , \theta_{q_1})$. Since $\vphi(p) = 1/k = \theta(q_1)$ and since both $M^\vphi$ and $N^\theta$ are factors, taking $k \times k$ matrices, we find a state preserving isomorphism $(M,\vphi) \cong (N,\theta)$.
\end{proof}

\begin{proof}[Proof of Corollary \ref{corB}]
Let $\mu, \nu \in \mathcal S(\R)$ such that $\Lambda(\mu_a) = \Lambda(\nu_a) =: \Lambda$ and $\mathcal C(\mu_c \ast \delta_\Lambda) = \mathcal C(\nu_c \ast \delta_{\Lambda})$. Let $m, n : \R \to \N \cup \{+\infty\}$ be any symmetric Borel multiplicity functions. Then we have $\mathcal C(\mu \ast \delta_\Lambda) = \mathcal C(\nu \ast \delta_\Lambda)$. By Theorem \ref{thm.isom}, there is a state preserving isomorphism $\Gamma(\mu,m)\dpr \cong \Gamma(\nu,n)\dpr$.

Conversely, let $\mu, \nu \in \mathcal S(\R)$ and $m, n : \R \to \N \cup \{+\infty\}$ be any symmetric Borel multiplicity functions such that $\Gamma(\mu, m)\dpr \cong \Gamma(\nu, n)\dpr$. By Theorem \ref{thmA}, we have that $\mathcal C(\bigvee_{k \geq 1} \mu^{\ast k}) = \mathcal C(\bigvee_{k \geq 1} \nu^{\ast k})$. Since for every $k \geq 1$, we have $\mu_c^{\ast k} \prec \mu_c$ and $\nu_c^{\ast k} \prec \nu_c$, it follows that
\begin{align*}
\mathcal C(\mu_c \ast \delta_{\Lambda(\mu_a)} \vee \delta_{\Lambda(\mu_a)}) &= \mathcal C(\bigvee_{k \geq 1} \mu^{\ast k}) \\
& = \mathcal C(\bigvee_{k \geq 1} \nu^{\ast k}) \\
&= \mathcal C(\nu_c \ast \delta_{\Lambda(\nu_a)} \vee \delta_{\Lambda(\nu_a)}).
\end{align*}
This implies that $\Lambda(\mu_a) = \Lambda(\nu_a)$ and $\mathcal C(\mu_c \ast \delta_{\Lambda(\mu_a)}) = \mathcal C(\nu_c \ast \delta_{\Lambda(\nu_a)})$.
\end{proof}

\begin{proof}[Proof of Corollary \ref{corC}]
Put $(M, \varphi) := (\Gamma(\mu, m)\dpr, \varphi_{\mu, m})$. Let $\psi \in M_*$ be a faithful normal state such that $M^\psi$ is nonamenable. By Theorem \ref{thm.meta}, we find a nonzero partial isometry $v \in M$ such that $q = v v^* \in M^\psi$, $p = v^* v \in M^\vphi$, $q M^\psi q$ has no amenable direct summand and $\psi q = \lambda \, v \vphi v^*$ with $\lambda = \psi(q)/\vphi(p)$. It follows that $p M^\vphi p \cong q M^\psi q$ has no amenable direct summand. By Remark \ref{rem.centralizer}, this means that either $\mu_a$ has an atom different from $0$ or $\mu_a$ is concentrated on $\{0\}$ with $m(0)\geq 2$. Conversely, if $\mu_a$ satisfies these properties, it follows from Remark \ref{rem.centralizer} that the centralizer of the free quasi-free state is nonamenable.
\end{proof}

\begin{proof}[Proof of Corollary \ref{corD}]
By Corollary \ref{corC}, the von Neumann algebra $\Gamma(\lambda+\delta_0,1)\dpr$ has amenable centralizers while $\Gamma(\lambda+\delta_0,2)\dpr$ does not.
\end{proof}

\begin{example}\label{ex.many-measure}
Many different measures in the family $\cS(\R)$ of Corollary \ref{corB} can be constructed as follows. Let $K \subset \R$ be an \emph{independent} Borel set, meaning that every $n$-tuple of distinct elements in $K$ generates a free abelian group of rank $n$. By \cite[Theorems 5.1.4 and 5.2.2]{Ru62}, there exist compact independent $K \subset \R$ such that $K$ is homeomorphic to a Cantor set. Fix such a $K \subset \R$ and put $L = K \cup (-K)$. Also fix a countable subgroup $\Lambda < \R$.

For every continuous symmetric probability measure $\mu$ on $\R$ that is concentrated on $L$, define the measure class $\mutil$ on $\R$ given by
$$\mutil = \bigvee_{x \in \Lambda, n \geq 1} (x + \mu^{\ast n}) \; .$$
By construction, each $\mutil$ is a continuous symmetric measure class on $\R$ that is invariant under translation by $\Lambda$ and that satisfies $\mutil \ast \mutil \prec \mutil$.

Given continuous symmetric probability measures $\mu_1$ and $\mu_2$ that are concentrated on $L$, we claim that $\cC(\widetilde{\mu_1}) = \cC(\widetilde{\mu_2})$ if and only if $\cC(\mu_1) = \cC(\mu_2)$. One implication is obvious. The other implication is a consequence of the following result contained in \cite[Corollary 1]{LP97}~: if $\eta_1$ and $\eta_2$ are concentrated on $L$ and $\eta_1 \perp \eta_2$, then also $\eta_1 \perp (x + \eta_2^{\ast k})$ for all $x \in \R$ and all $k \geq 1$.

Choosing $\Lambda$ to be a nontrivial subgroup of $\R$ and applying Corollary \ref{corB}, for all continuous symmetric probability measures $\mu_1$ and $\mu_2$ concentrated on the Cantor set $L$, we find that
$$\Gamma(\widetilde{\mu_1} \vee \delta_\Lambda,m_1)\dpr \cong \Gamma(\widetilde{\mu_2} \vee \delta_\Lambda,m_2)\dpr \quad\text{iff}\quad \cC(\mu_1) = \cC(\mu_2) \; .$$

Adding the Lebesgue measure to $\mutil$, we claim that we also have
$$\Gamma(\lambda \vee \widetilde{\mu_1} \vee \delta_\Lambda,m_1)\dpr \cong \Gamma(\lambda \vee \widetilde{\mu_2} \vee \delta_\Lambda,m_2)\dpr \quad\text{iff}\quad \cC(\mu_1) = \cC(\mu_2) \; .$$
By \cite[Corollary 8.6]{Sh97b}, for all these free Araki--Woods factors, the $\tau$-invariant equals the usual topology on $\R$, so that they cannot be distinguished by Connes' invariants.

To prove the claim, define $L_n$ as the $n$-fold sum $L_n = L + \cdots + L$ and put $S = \bigcup_{n \geq 1} L_n$. Below we prove that $\lambda(S) = 0$. The claim then follows from Corollary \ref{corB}~: if $\cC(\lambda \vee \widetilde{\mu_1}) = \cC(\lambda \vee \widetilde{\mu_2})$, restricting to $S$, we get that $\cC(\widetilde{\mu_1}) = \cC(\widetilde{\mu_2})$. As proven above, this implies that $\cC(\mu_1) = \cC(\mu_2)$.

It remains to prove that $\lambda(L_n) = 0$ for all $n$. If for some $n \geq 1$, we have $\lambda(L_n) > 0$, then $L_{2n} = L_n - L_n$ contains a neighborhood of $0$. Every nonzero $x \in L_{2n}$ can be uniquely written as $x = \al_1 y_1 + \cdots + \al_k y_k$ with $k \geq 1$, $y_1,\ldots,y_k$ distinct elements in $K$ and $\al_i \in \Z \setminus \{0\}$ with $|\al_i| \leq 2n$ for all $i$. So if $x \in L_{2n}$ is nonzero, we have that $(2n+1)x \not\in L_{2n}$. Therefore, $L_{2n}$ does not contain a neighborhood of $0$ and it follows that $\lambda(L_n) = 0$ for all $n \geq 1$.
\end{example}

\section{Proof of Theorem~\ref{thmE}}

To prove Theorem \ref{thmE}, we combine \cite[Theorem 4.3]{HU15} and Theorem \ref{thm-key} with the following lemma. Whenever $\theta$ is a faithful normal state on a von Neumann algebra $M$, we denote by $M_{\ap,\theta}$ the von Neumann subalgebra of $M$ generated by the almost periodic part of $(\sigma_t^\theta)$.

\begin{lem}\label{lem.centr-free-product}
For $i=1,2$, let $(M_i,\vphi_i)$ be von Neumann algebras with a faithful normal state. Denote by $(M,\vphi) = (M_1,\vphi_1)\ast (M_2,\vphi_2)$ their free product. Denote by $\rE_{M_1} : M \recht M_1$ the unique $\vphi$-preserving conditional expectation. Let $\theta_1$ be a faithful normal state on $M_1$ and define $\theta = \theta_1 \circ \rE_{M_1}$. Let $q \in M^\theta$ be a projection.

There exist projections $q_0,q_1,\ldots$ with $q_0 \in M_1^{\theta_1}$ and $q_i \in M^\theta$ for all $i \geq 1$ such that
\begin{itemize}
\item [$(\rm i)$] $\sum_{i=0}^\infty q_i = q$,
\item [$(\rm ii)$] $q_0 M_{\ap,\theta} q_0 = q_0 M_{1,\ap,\theta_1} q_0$ and $q_0 M^\theta q_0 = q_0 M_1^{\theta_1} q_0$,
\item [$(\rm iii)$] for every $i \geq 1$, there exists a partial isometry $v_i \in M$ with $v_i v_i^* = q_i$, $v_i^* v_i \in M^\vphi$ and
$$\frac{1}{\theta(q_i)} \, \theta q_i = \frac{1}{\vphi(v_i^* v_i)} \, v_i \vphi v_i^* \; .$$
\end{itemize}
\end{lem}

\begin{proof}
Fix standard representations $M_i \subset \rB(H_i)$. For every faithful normal state $\mu$ on $M_i$, denote by $\xi_\mu \in H_i$ the canonical unit vector that implements $\mu$.

Define $u_t = [\rD\theta_1 : \rD\vphi_1]_t \in \mathcal U(M_1)$. Note that also $[\rD\theta : \rD\vphi]_t = [\rD \theta_1 \circ \rE : \rD \vphi_1 \circ \rE]_t = u_t$ for all $t \in \R$. Let $e_1,e_2,\ldots$ be a maximal sequence of nonzero projections in $M_1^{\theta_1}$ such that $e_i e_j = 0$ whenever $i \neq j$ and such that for every $i \geq 1$, there exists a partial isometry $w_i \in M_1$ and a $\lambda_i > 0$ with $w_i w_i^* = e_i$ and $u_t \sigma_t^{\vphi_1}(w_i) = \lambda_i^{\ri t} w_i$ for all $t \in \R$. Define $e_0 = 1 - \sum_{i=1}^\infty e_i$. Then $e_0 \in M_1^{\theta_1}$. By construction, the unitary representation $(U_t)_{t \in \R}$ on $H_1$ given by $U_t (x \xi_{\vphi_1}) = u_t \sigma_t^{\vphi_1}(x) \xi_{\vphi_1}$ for all $x \in M_1$ is weakly mixing on $e_0 H_1$.

For $i = 1,2$, define $\Hcirc_i = H_i \ominus \C \xi_{\vphi_i}$. For every $k \geq 1$, define the Hilbert space
$$K_k = H_1 \ot \underbrace{\Hcirc_2 \ot \Hcirc_1 \ot \hspace{8mm} \cdots \hspace{8mm} \ot \Hcirc_1 \ot \Hcirc_2}_{\text{$k$ times $\Hcirc_2$ and $k-1$ times $\Hcirc_1$, alternatingly}} \ot H_1 \; .$$
We can then identify the standard Hilbert space $H$ for $M$ with
$$H = H_1 \oplus \bigoplus_{k = 1}^\infty K_k \; .$$
Under this identification, $\xi_\vphi = \xi_{\vphi_1} \in H_1$ and $\xi_\theta = \xi_{\theta_1} \in H_1$. Denote by $(V_t)_{t \in \R}$ the unitary representation on $H_1$ given by $V_t(x \xi_{\theta_1}) = \sigma_t^{\vphi_1}(x) u_t^* \xi_{\theta_1}$ for all $x \in M_1$. Under the above identification of $H$, we get that
$$\Delta_{\theta}^{\ri t} = \Delta_{\theta_1}^{\ri t} \oplus \bigoplus_{k=1}^\infty \Bigl( U_t \ot \Delta_{\vphi_2}^{\ri t} \ot \Delta_{\vphi_1}^{\ri t} \ot \cdots \ot \Delta_{\vphi_1}^{\ri t} \ot \Delta_{\vphi_2}^{\ri t} \ot V_t \Bigr) \; .$$
Since $(U_t)_{t \in \R}$ is weakly mixing on $e_0 H_1$, we conclude that $(\Delta_{\theta}^{\ri t})_{t \in \R}$ is weakly mixing on $e_0 H \ominus e_0 H_1$.
It follows that
\begin{equation}\label{eq.cont-ap}
e_0 M_{\ap,\theta} e_0 = e_0 M_{1,\ap,\theta_1} e_0 \quad\text{and}\quad e_0 M^\theta e_0 = e_0 M_1^{\theta_1} e_0 \; .
\end{equation}

Let $q_1,q_2,\ldots$ be a maximal sequence of nonzero projections in $q M^\theta q$ such that $q_i q_j = 0$ if $i \neq j$ and such that statement~$(\rm iii)$ in the lemma holds for every $i \geq 1$. Define $q_0 = q - \sum_{i=1}^\infty q_i$. Then $q_0 \in M^\theta$. We prove that $q_0 \leq e_0$. Once this is proven, it follows from \eqref{eq.cont-ap} that $q_0 \in M_1^{\theta_1}$ and that $q_0 M_{\ap,\theta} q_0 = q_0 M_{1,\ap,\theta_1} q_0$, so that the lemma follows.

If $q_0 \not\leq e_0$, we find $j \geq 1$ such that $q_0 e_j \neq 0$. Then the polar part $v$ of $q_0 w_j$ is a nonzero partial isometry in $M$ satisfying $vv^* \leq q_0$ and  $u_t \sigma_t^{\vphi_1}(v) = \lambda_j^{\ri t} v$ for all $t \in \R$. So, the projection $vv^*$ could be added to the sequence $q_1,q_2,\ldots$, contradicting its maximality. Therefore, $q_0 \leq e_0$ and the lemma is proved.
\end{proof}

Theorem \ref{thmE} will be an immediate consequence of the following more technical proposition that will also be used in Section \ref{sec.further} below.

\begin{prop}\label{prop.nonamen-centr-free-product}
For $i=1,2$, let $(M_i,\vphi_i)$ be von Neumann algebras with a faithful normal state. Denote by $(M,\vphi) = (M_1,\vphi_1)\ast (M_2,\vphi_2)$ their free product and by $\rE_{M_i} : M \recht M_i$ the unique $\vphi$-preserving conditional expectation. Let $\psi$ be a faithful normal state on $M$. Define the set of projections $\cP \subset M^\psi$ given by $\cP = \cP_1 \cup \cP_2 \cup \cP_3$ where
\begin{itemize}
\item for $i=1,2$, $\cP_i$ consists of the projections $q \in M^\psi$ for which there exists a partial isometry $v \in M$ and a faithful normal state $\theta_i$ on $M_i$ with $v^* v = q$, $e = vv^* \in M_i^{\theta_i}$,
    \begin{align*}
    & \frac{1}{\psi(q)} \, v \psi v^* = \frac{1}{\theta_i(e)} \, (\theta_i \circ \rE_{M_i}) e \quad , \\ & v M_{\ap,\psi} v^* = e M_{i,\ap,\theta_i} e \quad\text{and}\quad v M^\psi v^* = e M_i^{\theta_i} e \; ,
    \end{align*}
\item $\cP_3$ consists of the projections $q \in M^\psi$ for which there exists a partial isometry $v \in M$ with $v^* v = q$, $e = vv^* \in M^\vphi$,
$$\frac{1}{\psi(q)} \, v \psi v^* = \frac{1}{\vphi(e)} \, \vphi e \quad\text{and}\quad v M^\psi v = e M^\vphi e \; .$$
\end{itemize}
If $q \in M^\psi$ is a  projection such that $q M^\psi q$ has no amenable direct summand, then $q$ can be written as a sum of projections in $\cP$.
\end{prop}
\begin{proof}
Let $\psi$ be a faithful normal state on $M$ and $q \in M^\psi$ a projection such that $q M^\psi q$ has no amenable direct summand. It suffices to prove that $q$ dominates a nonzero projection in $\cP$, since then a maximality argument can be applied.

Fix any nonzero finite trace projection $r_0 \in \rL_\psi(\R)$ and put $r = \Pi_{\vphi,\psi}(q r_0)$. Define the von Neumann subalgebra $Q \subset r \core_\vphi(M) r$ given by
$$Q = \Pi_{\vphi,\psi}\bigl(\rL_\psi(\R) q r_0 \vee q M^\psi q r_0 \bigr) \; .$$
Note that $Q$ has no amenable direct summand. By \cite[Theorem 4.3]{HU15},
$$\text{either}\quad Q' \cap r \core_\vphi(M) r \prec_{\core_\vphi(M)} \rL_\vphi(\R) \quad\text{or}\quad Q \prec_{\core_\vphi(M)} \core_\vphi(M_i) \;\;\text{for $i=1$ or $i=2$.}$$
Since $\Pi_{\vphi,\psi}(\rL_\psi(\R) q r_0)$ belongs to both $Q$ and $Q' \cap r \core_\vphi(M) r$, it follows that
$$\Pi_{\vphi,\psi}(\rL_\psi(\R) q r_0) \prec_{\core_\vphi(M)} \core_{\vphi}(M_i)$$
for $i=1$ or $i=2$.

By Theorem \ref{thm-key}, we find a faithful normal state $\theta_i$ on $M_i$ and a partial isometry $v \in M$ such that $q_0 = v^* v$ is a nonzero projection in $q M^\psi q$, $p = vv^*$ belongs to $M^{\theta_i \circ \rE_{M_i}}$ and
$$\frac{1}{\psi(q_0)} \, v \psi v^* = \frac{1}{\theta_i(\rE_{M_i}(p))} \, (\theta_i \circ \rE_{M_i}) p \; .$$
Write $\theta = \theta_i \circ \rE_{M_i}$. By Lemma \ref{lem.centr-free-product}, we either find a nonzero projection $e \leq p$ such that $e \in M_i^{\theta_i}$ and $e M_{\ap,\theta} e = e M_{i,\ap,\theta_i} e$ and $eM^\theta e = e M_i^{\theta_i} e$, or we find a nonzero projection $p_0 \in p M^\theta p$ and a partial isometry $w \in M$ such that $ww^* = p_0$, $e = w^* w$ belongs to $M^\vphi$ and
$$\frac{1}{\theta(p_0)} \, w^* \theta w = \frac{1}{\vphi(e)} \, \vphi e \; .$$
In the first case, we get that the projection $v^* e v \leq q$ belongs to $\cP_i$, while in the second case, the projection $v^* p_0 v \leq q$ belongs to $\cP_3$.
\end{proof}

\begin{proof}[Proof of Theorem \ref{thmE}]
Denote by $\rE_{M_i} : M \recht M_i$ the unique $\vphi$-preserving conditional expectation. If $M^\vphi$ is nonamenable, then obviously, $M$ does not have all its centralizers amenable. If $M_i$ admits a faithful normal state $\theta_i$ such that $M_i^{\theta_i}$ is nonamenable, then $\theta_i \circ \rE_{M_i}$ is a faithful normal state on $M$ with $M_i^{\theta_i} \subset M^{\theta_i \circ \rE_{M_i}}$, so that again, $M$ does not have all its centralizers amenable.

Conversely, assume that $\psi$ is a faithful normal state on $M$ such that $M^\psi$ is nonamenable. Take a nonzero projection $q \in M^\psi$ such that $q M^\psi q$ has no amenable direct summand. By \ref{prop.nonamen-centr-free-product}, we either find $i \in \{1,2\}$ and a faithful normal state $\theta_i$ on $M_i$ such that $M_i^{\theta_i}$ is nonamenable, or we find that $M^\vphi$ is nonamenable.
\end{proof}

\section{Further structural results and proof of Theorem \ref{thmF}}\label{sec.further}

We start by showing that the invariant of Theorem \ref{thmA} is not a complete invariant for the family of free Araki--Woods factors $\Gamma(\mu, m)\dpr$ arising from finite symmetric Borel measures $\mu$ on $\R$ whose atomic part $\mu_a$ is nonzero and not supported on $\{0\}$.

\begin{thm}\label{thm-invariant}
Let $\Lambda < \R$ be any countable subgroup such that $\Lambda \neq \{0\}$ and denote by $\delta_\Lambda$ a finite atomic measure on $\R$ whose set of atoms equals $\Lambda$. Let $\eta$ be any continuous finite symmetric Borel measure on $\R$ such that $\mathcal C(\eta) = \mathcal C(\eta \ast\delta_\Lambda)$ and such that the measures $(\eta^{\ast k})_{k \geq 1}$ are pairwise singular.

Put $\mu = \delta_\Lambda + \eta$ and $\nu = \delta_\Lambda + \eta + \eta \ast \eta$. Then,
$$\Gamma(\mu,1)\dpr \not\cong \Gamma(\nu,1)\dpr \quad \text{and}  \quad \mathcal C(\bigvee_{k \geq 1} \mu^{\ast k}) = \mathcal C(\bigvee_{k \geq 1} \nu^{\ast k}).$$
\end{thm}

\begin{proof}
By construction, we have $\mathcal C(\bigvee_{k \geq 1} \mu^{\ast k}) = \mathcal C(\bigvee_{k \geq 1} \nu^{\ast k})$. We denote $M := \Gamma(\mu,1)\dpr$ and $N := \Gamma(\nu,1)\dpr$. We denote by $Q \subset N$ the canonical von Neumann subalgebra given by $Q := \Gamma(\delta_\Lambda + \eta,1)\dpr$. Put $\varphi := \varphi_{\mu,1}$ and $\psi := \varphi_{\nu,1}$. Observe that the inclusion $Q \subset N$ is globally invariant under the modular automorphism group $\sigma^{\psi}$.

Assume by contradiction that $M \cong N$. By Theorem \ref{thmA}, there exists a state preserving isomorphism $\pi : (M ,\varphi) \recht (N, \psi)$ of $M$ onto $N$. Then, $\pi$ extends to a unitary operator $U : \rL^2(M,\vphi) \recht \rL^2(N,\psi)$ satisfying $U \Delta_\vphi U^* = \Delta_\psi$. Define the real Hilbert space
$$H_\R^\mu := \left\{ f \in \rL^2_\C(\R, \mu) : f(-s) = \overline {f(s)} \text{ for } \mu\text{-almost every } s \in \R\right\}$$
and the orthogonal representation
$$U^\mu : \R \curvearrowright H_\R^\mu : U_s^\mu(f)(t) = \exp({\rm i}st) f(t) \; .$$
Denote by $s(\xi) := \ell(\xi) + \ell(\xi)^*$, $\xi \in H_\R^\mu$, the canonical semicircular elements that generate $M$ and satisfy $\sigma_t^\vphi(s(\xi)) = s(U_t^\mu \xi)$.
By construction, $\mathcal C(\Delta_{\psi} |_{\rL^2(N) \ominus \rL^2(Q)}) = \bigcap_{k \geq 2} \mathcal C(\eta^{\ast k})$. Since $\mu$ is singular w.r.t.\ $\eta^{\ast k}$ for all $k \geq 2$, it follows that $\pi(s(\xi)) \in Q$ for all $\xi \in H_\R^\mu$. But then, $\pi(M) \subset Q$, which is impossible because $\pi$ is surjective.
\end{proof}

Note that Example \ref{ex.many-measure} provides many measures $\eta$ satisfying the assumptions of Theorem \ref{thm-invariant}.

We could not prove or disprove that the measure class $\mathcal C(\mu_c * \delta_\Lambda(\mu_a))$ is an invariant for the family of free Araki--Woods factors $\Gamma(\mu, m)\dpr$ arising from finite symmetric Borel measures $\mu$ on $\R$ whose atomic part $\mu_a$ is nonzero and not supported on $\{0\}$.

Using Theorem \ref{thm.meta} in combination with the results of \cite{BH16}, we can also clarify the relation between free Araki--Woods factors and free products of amenable von Neumann algebras. Combining \cite[Theorems 2.11 and 4.8]{Sh96}, it follows that every almost periodic free Araki--Woods factor is isomorphic with a free product of von Neumann algebras of type~I. Conversely, by \cite{Ho06}, many free products of type~I von Neumann algebras are isomorphic with free Araki--Woods factors.
In \cite[Theorem 4.6]{Dy92}, it is proved that a free product of two amenable von Neumann algebras w.r.t.\ faithful normal traces is always isomorphic to the direct sum of an interpolated free group factor and a finite dimensional algebra. It is therefore tempting to believe that every type~III factor arising as a free product of amenable von Neumann algebras w.r.t.\ faithful normal states is a free Araki--Woods factor. The following example shows however that this is almost never the case if one of the states fails to be almost periodic.

\begin{thm}
Let $(P,\theta) = \ast_n (P_n,\theta_n)$ be a free product of amenable von Neumann algebras. Assume that the centralizer $P^\theta$ has no amenable direct summand and that at least one of the $\theta_n$ is not almost periodic.

Then $P$ is not isomorphic to a free Araki--Woods factor. Even more: there is no faithful normal homomorphism $\pi$ of $P$ into a free Araki--Woods factor $M$ such that $\pi(P) \subset M$ is with expectation.

The same conclusions hold if $(P,\theta)$ is any von Neumann algebra with a faithful normal state satisfying the following three properties: the centralizer $P^\theta$ has no amenable direct summand, $\theta$ is not almost periodic and $P$ is generated by a family of amenable von Neumann subalgebras $P_n \subset P$ that are globally invariant under the modular automorphism group $(\sigma_t^\theta)$.
\end{thm}
\begin{proof}
Let $(M,\vphi)$ be a free Araki--Woods factor with its free quasi-free state. Let $(P,\theta)$ be a von Neumann algebra with a faithful normal state $\theta$ such that the centralizer $P^\theta$ has no amenable direct summand and such that $P$ is generated by a family of amenable von Neumann subalgebras $P_n \subset P$ that are globally invariant under the modular automorphism group $(\sigma_t^\theta)$. Let $\pi : P \recht M$ be a normal homomorphism and $\rE : M \recht \pi(P)$ a faithful normal conditional expectation. We prove that $\theta$ is almost periodic.

Define the faithful normal state $\psi \in M_*$ given by $\psi = \theta \circ \pi^{-1} \circ \rE$. Since $\pi(P^\theta) \subset M^\psi$, we get that $M^\psi$ has no amenable direct summand. By Theorem \ref{thm.meta}, we find partial isometries $v_n \in M$ such that $q_n = v_n v_n^* \in M^\psi$, $p_n = v_n^* v_n \in M^\vphi$, $\sum_n q_n = 1$ and
$$q_n \, [\rD\psi : \rD \vphi]_t = \lambda_n^{\ri t} \, v_n \, \sigma_t^\vphi(v_n^*)$$
where $\lambda_n = \psi(q_n)/\vphi(p_n)$.

Replacing $(M,\vphi)$ by the free product of $(M,\vphi)$ and the appropriate almost periodic free Araki--Woods factor, we still get a free Araki--Woods factor and we may assume that $M^\vphi$ is a factor and that each $\lambda_n$ is an eigenvalue for the free quasi-free state. We can then choose partial isometries $w_n \in M$ such that $\sigma_t^\vphi(w_n) = \lambda_n^{-\ri t} w_n$, $w_n w_n^* = p_n$ and such that $e_n = w_n^* w_n$ belongs to $M^\vphi$ with $\sum_n e_n = 1$. So, we find that
$$q_n \, [\rD\psi : \rD \vphi]_t = v_n w_n \, \sigma_t^\vphi(w_n^* v_n^*)$$
for all $n$ and all $t \in \R$. We conclude that $v = \sum_n v_n w_n$ is a unitary in $M$ satisfying $[\rD\psi : \rD\vphi]_t = v \, \sigma_t^\vphi(v^*)$. This means that $\vphi = \psi \circ \Ad v$ and that the homomorphism $\eta = \Ad v^* \circ \pi$ satisfies $\sigma_t^\vphi \circ \eta = \eta \circ \sigma_t^\theta$.

So for every $n$, the subalgebra $\eta(P_n) \subset M$ is amenable and globally invariant under the modular automorphism group $(\sigma_t^\vphi)$. By \cite[Theorem 4.1]{BH16}, it follows that $\eta(P)$ lies in the almost periodic part of $(M,\vphi)$. This implies that the restriction of $(\sigma_t^\theta)$ to $P_n$ is almost periodic. Since this holds for every $n$, we conclude that $\theta$ is almost periodic.
\end{proof}

From Proposition \ref{prop.nonamen-centr-free-product}, we get the following rigidity results for free product von Neumann algebras. Roughly, the result says that an arbitrary free product of a ``very much non almost periodic'' $M_1$ with an almost periodic $M_2$ remembers the almost periodic part $M_2$ up to amplification.

Recall that a faithful normal positive functional $\vphi$ on a von Neumann algebra $M$ is said to be \emph{weakly mixing} if the unitary representation $\sigma_t^\vphi(\,\cdot\,)$ on $\rL^2(M,\vphi) \ominus \C 1$ is weakly mixing, and that $\vphi$ is said to be \emph{almost periodic} if the unitary representation $\sigma_t^\vphi(\,\cdot\,)$ on $\rL^2(M,\vphi)$ is almost periodic.

\begin{prop}\label{prop.rigidity-free-products}
For $i = 1,2$, let $(M_i,\vphi_i)$ and $(N_i,\psi_i)$ be von Neumann algebras with faithful normal states. Denote by $(M,\vphi) = (M_1,\vphi_1) \ast (M_2,\vphi_2)$ and $(N,\psi) = (N_1,\psi_1) \ast (N_2,\psi_2)$ their free products. Assume that
\begin{itemize}
\item $M_1$ and $N_1$ have all their centralizers amenable and $\vphi_1$, $\psi_1$ are weakly mixing states,
\item $M_2^{\vphi_2}$ and $N_2^{\psi_2}$ have no amenable direct summand and $\vphi_2$, $\psi_2$ are almost periodic.
\end{itemize}
If $M \cong N$, there exist nonzero projections $e \in M_2$ and $q \in N_2$ such that $eM_2 e \cong q N_2 q$.
\end{prop}

\begin{proof}
Whenever $\mu$ is a faithful normal state on $M$ and $q \in M^\mu$ is a projection such that $q M^\mu q$ has no amenable direct summand, we can apply Proposition \ref{prop.nonamen-centr-free-product}. Since $M_1$ has all its centralizers amenable, the set $\cP_1$ in Proposition \ref{prop.nonamen-centr-free-product} equals $\{0\}$. Since $\vphi_1$ is weakly mixing, the almost periodic part of a state of the form $\theta_2 \circ \rE_{M_2}$ (and, in particular, of $\vphi$) is contained in $M_2$. It thus follows from Proposition \ref{prop.nonamen-centr-free-product} that there exist sequences of projections $q_i \in M^\mu$ and $e_i \in M_2$, as well as partial isometries $v_i \in M$ and faithful normal positive functionals $\theta_i$ on $e_i M_2 e_i$ such that $v_i^* v_i = q_i$, $v_i v_i^* = e_i$, $\sum_{i=0}^\infty q_i = q$ and $v_i \mu v_i^*$ equals $\theta_i \circ \rE_{M_2}$ on $e_i M e_i$ for all $i \geq 0$.

Let $\pi : M \recht N$ be an isomorphism of $M$ onto $N$. We first apply the result in the first paragraph to $\mu = \psi \circ \pi$. Since $\psi_1$ is weakly mixing, $M^\mu = \pi^{-1}(N_2^{\psi_2})$. We find nonzero projections $q \in N_2^{\psi_2}$ and $e \in M_2$, a partial isometry $v \in M$ and a faithful normal positive functional $\theta$ on $e M_2 e$ such that $v^* v = \pi^{-1}(q)$, $vv^* = e$ and $v \mu v^* = \theta \circ \rE_{M_2}$ on $e M e$. Since $\vphi_1$ is weakly mixing, the almost periodic part of $\theta \circ \rE_{M_2}$ equals $e M_{2,\ap,\theta} e$. Since $\psi_1$ is weakly mixing and $\psi_2$ is almost periodic, the almost periodic part of $\mu$ equals $\pi^{-1}(N_2)$. It follows that
\begin{equation}\label{eq.step1}
v \pi^{-1}(N_2) v^* = e M_{2,\ap,\theta} e \; .
\end{equation}

By \cite[Lemma 2.1]{HU15}, every projection in $M_2$ is equivalent, inside $M_2$, with a projection in $M_2^{\vphi_2}$. So conjugating $e$ and $\theta$, we may assume that $e \in M_2^{\vphi_2}$. We then apply the result of the first paragraph of the proof to the free product $N = N_1 \ast N_2$, the faithful normal state $\mu' = \vphi \circ \pi^{-1}$ and the projection $f = \pi(e)$ in $N^{\mu'}$. Since $\vphi_1$ is weakly mixing, we have $N^{\mu'} = \pi(M_2^{\vphi_2})$. We thus find projections $e_i \in e M_2^{\vphi_2} e$ summing up to $e$, projections $p_i \in N_2$, partial isometries $w_i \in N$ and faithful normal positive functionals $\Om_i$ on $p_i N_2 p_i$ such that $w_i^* w_i = \pi(e_i)$, $w_i w_i^* = p_i$ and $w_i \mu' w_i^* = \Om_i \circ \rE_{N_2}$ on $p_i N p_i$ for all $i$. In particular, $\Om(p_i) = \vphi_2(e_i)$ and $\sum_i \Om(p_i) = \vphi_2(e)$.

Define the projection $p \in \rB(\ell^2(\N)) \ovt N_2$ given by $p = \sum_i e_{ii} \ot p_i$. Define the faithful normal positive functional $\Om$ on $p(\rB(\ell^2(\N)) \ovt N_2)p$ given by
$$\Om(T) = \sum_i \Om_i(T_{ii}) \; .$$
Finally define $W \in \ell^2(\N) \ot N$ given by $W = \sum_i e_i \ot w_i$. It follows that $W^* W = \pi(e)$, $W W^* = p$ and
$$W \mu' W^* = \Om \circ \rE_{\rB(\ell^2(\N)) \ovt N_2} \quad\text{on}\;\; p(\rB(\ell^2(\N)) \ovt N)p \; .$$
As above, $N_{\ap,\mu'} = \pi(M_{\ap,\vphi}) = \pi(M_2)$. Since $\psi_1$ is weakly mixing, the almost periodic part of the functional $\Om \circ \rE_{\rB(\ell^2(\N)) \ovt N_2}$ on $p(\rB(\ell^2(\N)) \ovt N)p$ is contained in $p(\rB(\ell^2(\N)) \ovt N_2)p$. It follows that
\begin{equation}\label{eq.step2}
W \pi(e M_2 e) W^* \subset p(\rB(\ell^2(\N)) \ovt N_2)p \; .
\end{equation}
Write $V = W \pi(v)$. Then $V \in \ell^2(\N) \ot N$ is a partial isometry with $VV^* = p \in \rB(\ell^2(\N)) \ovt N_2$ and $V^* V = q \in N_2$. Using \eqref{eq.step1} and \eqref{eq.step2}, we find that
\begin{equation}\label{eq.step3}
V q N_2 q V^* = W \pi(e M_{2,\ap,\theta} e) W^* \subset W \pi(e M_2 e)W^* \subset p (\rB(\ell^2(\N)) \ovt N_2) p \; .
\end{equation}
Since $N_2^{\psi_2}$ has no amenable direct summand, it follows in particular that $N_2$ is diffuse. Then \eqref{eq.step3} implies that $V \in \ell^2(\N) \ot N_2$ and therefore, all the inclusions in \eqref{eq.step3} are equalities. In particular, $e M_{2,\ap,\theta} e = e M_2 e$ so that \eqref{eq.step1} implies that $q N_2 q \cong e M_2 e$. This concludes the proof of the proposition.
\end{proof}

Combining Propositions \ref{prop.corner-free-product-aw} and \ref{prop.rigidity-free-products}, we can easily prove Theorem \ref{thmF}.

\begin{proof}[Proof of Theorem \ref{thmF}]
Put $(M,\vphi) = (\Gamma(\mu,+\infty)\dpr,\vphi_{\mu,+\infty})$ as in the formulation of the Proposition. Then, $\vphi$ is weakly mixing and by Corollary \ref{corC}, the free Araki--Woods factor $M$ has all its centralizers amenable. For $i=1,2$, let $(A_i,\psi_i)$ be von Neumann algebras with almost periodic faithful normal states having a nonamenable factorial centralizer $A_i^{\psi_i}$.

If the free products $(M_i,\vphi_i) = (M,\vphi) \ast (A_i,\psi_i)$ satisfy $M_1 \cong M_2$, it follows from Proposition \ref{prop.rigidity-free-products} that there exist nonzero projections $p_i \in A_i$ such that $p_1 A_1 p_1 \cong p_2 A_2 p_2$. In the first case, where the $A_i$ are ${\rm II_1}$ factors, this implies that $A_1 \cong A_2^t$ for some $t > 0$. In the second case, where the $A_i$ are type III factors, this implies that $A_1 \cong A_2$.

For the converse, first assume that the $(A_i,\psi_i)$ are ${\rm II_1}$ factors with their tracial states and $A_1 \cong A_2^t$ for some $t > 0$. Take nonzero projections $p_i \in A_i$ such that $p_1 A_1 p_1 \cong p_2 A_2 p_2$. By the uniqueness of the trace, we have $(p_1 A_1 p_1,(\psi_1)_{p_1}) \cong (p_2 A_2 p_2,(\psi_2)_{p_2})$. It then follows from Proposition \ref{prop.corner-free-product-aw} that $p_1 M_1 p_1 \cong p_2 M_2 p_2$. Since the $M_i$ are type ${\rm III}$ factors, this further implies that $M_1 \cong M_2$.

Finally assume that the $(A_i,\psi_i)$ are full type III factors with almost periodic states having a factorial centralizer $A_i^{\psi_i}$ and that $\pi : A_1 \recht A_2$ is an isomorphism of $A_1$ onto $A_2$. Denote by $\Gamma = \Sd(A_1) = \Sd(A_2)$ the $\Sd$-invariant of $A_1 \cong A_2$. Define $(B_i,\theta_i) = (\rB(\ell^2(\N)) \ovt A_i, \Tr \ot \psi_i)$. By \cite[Lemma 4.8]{Co74}, the weight $\theta_i$ on $B_i$ is a $\Gamma$-almost periodic weight. By \cite[Theorem 4.7]{Co74}, there exists a unitary $U \in B_2$ and a constant $\al > 0$ such that $\theta_2 \circ \Ad(U) \circ (\id \ot \pi) = \al \, \theta_1$. Since $A^{\psi_2}$ is a factor, after a unitary conjugacy of $\pi$, we find nonzero projections $p_i \in A_i^{\psi_i}$ such that $\pi(p_1) = p_2$ and $(\psi_2)_{p_2} \circ \pi = (\psi_1)_{p_1}$ on $p_1 A_1 p_1$. As in the previous paragraph, we can use Proposition \ref{prop.corner-free-product-aw} to conclude that $M_1 \cong M_2$.
\end{proof}

We finally consider two further structural properties of free Araki--Woods factors: the free absorption property and the structure of its continuous core. We say that a von Neumann algebra $M$ with a faithful normal state $\vphi$ has the \emph{free absorption property} if the free product $(N,\psi) = (M,\vphi) \ast (\rL(\F_\infty),\tau)$ satisfies $N \cong M$. One of the key results in \cite{Sh96} is the free absorption property for the almost periodic free Araki--Woods factors. In general, we get the following result.

\begin{prop}
Let $(M,\vphi) = (\Gamma(\mu,+\infty)\dpr,\vphi_{\mu,+\infty})$ be a free Araki--Woods factor with infinite multiplicity. Then $(M,\vphi)$ has the free absorption property if and only if the atomic part $\mu_a$ is nonzero.
\end{prop}
\begin{proof}
If $\mu(\{0\}) > 0$, then $(M,\vphi)$ freely splits off $(\rL(\F_\infty),\tau)$ and the free absorption property immediately holds. If $\mu(\{a\}) > 0$ for some $a \neq 0$, then $(M,\vphi)$ freely splits off an almost periodic free Araki--Woods factor of type III and the free absorption property follows from \cite[Theorem 5.4]{Sh96}. Conversely, if $\mu_a = 0$, it follows from Corollary \ref{corC} that $M$ has all its centralizers amenable. But then $M$ cannot have the free absorption property.
\end{proof}

One of the most intriguing isomorphism questions for free Araki--Woods factors, well outside the scope of our methods, is whether $\Gamma(\lambda,1)\dpr \cong \Gamma(\lambda + \delta_0,1)\dpr$~? In \cite[Theorem 4.8]{Sh97a}, it was shown that the continuous core of $\Gamma(\lambda,1)\dpr$ is isomorphic with $\rB(\ell^2(\N)) \ovt \rL(\F_\infty)$. We prove that the same holds for $\Gamma(\lambda + \delta_0,1)\dpr$. Note here that in \cite[Corollary 1.10]{Ha15}, it is proved that if $\mu$ is singular w.r.t.\ the Lebesgue measure $\lambda$, then the continuous core of $\Gamma(\mu,m)\dpr$ is never isomorphic with $\rB(\ell^2(\N)) \ovt \rL(\F_\infty)$. Under the stronger assumption that all convolution powers $\mu^{\ast n}$ are singular w.r.t.\ the Lebesgue measure, this was already shown in \cite{Sh02}.

\begin{prop}\label{prop.core-lambda-plus-delta-0}
The continuous core of $\Gamma(\lambda + \delta_0,1)\dpr$ is isomorphic with $\rB(\ell^2(\N)) \ovt \rL(\F_\infty)$.
\end{prop}
\begin{proof}
In \cite{Sh97a,Sh97b}, von Neumann algebras generated by $A$-valued semicircular elements are introduced. In the special case where $A$ is semifinite and equipped with a fixed faithful normal semifinite trace $\Tr$, this construction can be summarized as follows.

Let $H$ be a Hilbert $A$-bimodule, meaning that $H$ is a Hilbert space equipped with a normal homomorphism $A \recht \rB(H)$ and a normal anti-homomorphism $A \recht \rB(H)$ having commuting images. We denote the left and right action of $A$ on $H$ as $a \cdot \xi \cdot b$ for all $a,b \in A$, $\xi \in H$. Further assume that $S$ is an $A$-anti-bimodular involution on $H$. More precisely, $S$ is a closed, densely defined operator on $H$ such that $S(\xi) \in \rD(S)$ with $S(S(\xi)) = \xi$ for all $\xi \in \rD(S)$ and such that for all $\xi \in \rD(S)$ and all $a,b \in A$, we have $a \cdot \xi \cdot b \in \rD(S)$ and $S(a \cdot \xi \cdot b) = b^* \cdot S(\xi) \cdot a^*$. Define
$$\cF_A(H) = \rL^2(A,\Tr) \oplus \bigoplus_{k=1}^\infty \bigl(\underbrace{H \ot_A H \ot_A \cdots \ot_A H}_{\text{$k$ factors}}\bigr) \; .$$
A vector $\xi \in H$ is called right bounded if there exists a $\kappa > 0$ such that $\|\xi \cdot a\| \leq \kappa \|a\|_{2,\Tr}$ for all $a \in \cn_{\Tr}$. For every $\xi \in H$, there exists an increasing sequence of projections $p_n \in A$ such that $p_n \recht 1$ strongly and $\xi \cdot p_n$ is right bounded for all $n$. So, the subspace $H_0 \subset H$ defined as
$$H_0 = \{\xi \in \rD(S) : \xi \;\;\text{and}\;\; S(\xi)\;\;\text{are right bounded}\;\}$$
is dense. For every right bounded vector $\xi \in H$, we have a natural left creation operator $\ell(\xi) \in \rB(\cF_A(H))$. Then we define
$$\Phi(A,\Tr,H,S) = \bigl( A \cup \{\ell(\xi) + \ell(S(\xi))^* : \xi \in H_0 \}\bigr)\dpr \; .$$
There is a normal conditional expectation $\rE : \Phi(A,\Tr,H,S) \recht A$ given by $\rE(x) P = P x P$, where $P : \cF_A(H) \recht \rL^2(A,\Tr)$ is the orthogonal projection. By \cite[Proposition 5.2]{Sh97b}, we have that $\rE$ is faithful. By construction, if $A = \C 1$ and $\Tr(1) = 1$, and using the notation introduced before Lemma \ref{lem.model}, we find the free Araki--Woods factor $\Phi(\C 1,\Tr,H,S) \cong \Gamma(H,S)\dpr$.

The above construction can be applied to a normal completely positive map $\vphi : A \recht A$ satisfying $\Tr(\vphi(a) b) = \Tr(a \vphi(b))$ for all $a , b \in \cm_{\Tr}$. To such a map $\vphi$, we associate the Hilbert $A$-bimodule $H_\vphi$ by separation and completion of $A \ot \cn_{\Tr}$ with inner product
$$\langle a \ot_\vphi b , c \ot_\vphi d \rangle = \Tr(b^* \vphi(a^* c) d) \; .$$
We also define the anti-unitary involution $S(a \ot_\vphi b) = b^* \ot_\vphi a^*$. We denote the resulting von Neumann algebra $\Phi(A,\Tr,H,S)$ as $\Phi(A,\Tr,\vphi)$.

Given a trace preserving inclusion $(A,\Tr) \subset (D,\Tr)$, we denote by $\rE : D \recht A$ the unique $\Tr$-preserving conditional expectation and define $\psi : D \recht D : \psi(d) = \vphi(\rE(d))$ for all $d \in D$. Then, the functor $\Phi$ satisfies
\begin{equation}\label{eq.iso-amalg-Phi}
\Phi(A,\Tr,\vphi) \ast_A D \cong \Phi(D,\Tr,\psi)
\end{equation}
where the amalgamated free product is taken w.r.t.\ the canonical conditional expectations.

Let $(M,\vphi) = (\Gamma(\lambda,1)\dpr,\vphi_{\lambda,1})$. We can then reformulate \cite[Theorem 4.1]{Sh97a} as
\begin{equation}\label{eq.first-iso-core}
\core_\vphi(M) \cong \Phi(A,\Tr,\vphi)
\end{equation}
where $A = \rL^\infty(\R)$, $\Tr(f) = \int_\R f(x) \exp(-x) \, {\rm d}x$ and $\vphi : A \recht A$ is such that the associated $A$-bimodule $H$ is isomorphic with the coarse $A$-bimodule $\rL^2(\R^2)$ with anti-unitary involution $(S \xi)(x,y) = \overline{\xi(y,x)}$. Under the identification $\rL_\vphi(\R) \cong \rL^\infty(\R)$, the isomorphism in \eqref{eq.first-iso-core} respects the canonical conditional expectations $\core_\vphi(M) \recht \rL_\vphi(\R)$ and $\Phi(A,\Tr,\vphi) \recht A$.

By \cite[Theorem 4.8]{Sh97a}, we have in this particular case that $\Phi(A,\Tr,\vphi) \cong \rB(\ell^2(\N)) \ovt \rL(\F_\infty)$. Let now $(D,\Tr_D)$ be an arbitrary diffuse abelian von Neumann algebra with a faithful normal semifinite trace satisfying $\Tr(1) = +\infty$ and let $\psi : D \recht D$ be any normal completely positive map satisfying $\Tr_D(\psi(c) d) = \Tr_D(c \psi(d))$ for all $c,d \in \cm_{\Tr_D}$ such that the associated $D$-bimodule $H$ and anti-unitary involution $S$ are isomorphic with $\rL^2(D,\Tr_D) \ot \rL^2(D,\Tr_D)$ with $S(c \ot d) = d^* \ot c^*$. Since there exists an isomorphism $\al : D \recht A$ of $D$ onto $A$ satisfying $\Tr \circ \al = \Tr_D$, it follows that
\begin{equation}\label{eq.second-iso}
\begin{split}
\Phi(D,\Tr_D,\psi) &\cong \Phi(D,\Tr_D, \rL^2(D,\Tr_D) \ot \rL^2(D,\Tr_D),S) \\ &\cong \Phi(A,\Tr,\rL^2(A,\Tr) \ot \rL^2(A,\Tr),S) \\ &\cong \Phi(A,\Tr,\vphi) \cong \rB(\ell^2(\N)) \ovt \rL(\F_\infty) \; .
\end{split}
\end{equation}

Write $(M_1,\vphi_1) = (\Gamma(\lambda + \delta_0,1)\dpr,\vphi_{\lambda+\delta_0,1})$. Since $(M_1,\vphi_1) \cong (M,\vphi)\ast (B,\tau)$ for some diffuse abelian von Neumann algebra $B$ with faithful normal state $\tau$, it follows that
$$\core_{\vphi_1}(M_1) \cong \core_\vphi(M) \ast_{\rL_\vphi(\R)} (\rL_\vphi(\R) \ovt B) \; .$$
Since the isomorphism in \eqref{eq.first-iso-core} respects the conditional expectations, we conclude that
$$\core_{\vphi_1}(M_1) \cong \Phi(A,\Tr,\vphi) \ast_A (A \ovt B) \; ,$$
where the conditional expectation $A \ovt B \recht A$ is given by $\id \ot \tau$. Write $D = A \ovt B$ and define $\psi : D \recht D : \psi(d) = \vphi((\id \ot \tau)(d)) \ot 1$. By \eqref{eq.iso-amalg-Phi}, we get that
$$\Phi(A,\Tr,\vphi) \ast_A (A \ovt B) \cong \Phi(D, \Tr \ot \tau, \psi) \; .$$
Since $D$ is diffuse abelian and the $D$-bimodule associated with $\psi$ is isomorphic with the coarse $D$-bimodule, it follows from \eqref{eq.second-iso} that $\Phi(D,\Tr \ot \tau,\psi) \cong \rB(\ell^2(\N)) \ovt \rL(\F_\infty)$. So, the proposition is proved.
\end{proof}

%%%%%%%%%%

\bibliographystyle{plain}

\end{document}